\documentclass[12pt,draft]{article}

\usepackage{amsmath}
\usepackage{amssymb}
\usepackage{amsthm}
\usepackage{bbm}
\usepackage{MnSymbol}
\usepackage{xcolor}
\usepackage[T2A]{fontenc}
\usepackage[utf8]{inputenc}
\usepackage{enumerate}
\usepackage[hidelinks]{hyperref}
\hypersetup{
colorlinks = true,
linkcolor = blue,
linkbordercolor = {white},
citecolor = blue}

\newtheorem{theorem}{Theorem}

\newtheorem{lemma}{Lemma}
\newtheorem{cor}{Corollary}

\DeclareMathOperator{\1}{\mathbbm{1}}

\newcommand{\mmp}{\mathbb{P}}
\newcommand{\mn}{\mathbb{N}}
\newcommand{\me}{\mathbb{E}}
\newcommand{\eee}{{\rm e}}
\newcommand{\mr}{\mathbb{R}}

\begin{document}
\title{On intermediate levels of nested occupancy scheme in random environment generated by stick-breaking: the case of heavy tails}

\author{Oksana Braganets\footnote{Faculty of Computer Science and Cybernetics, Taras Shevchenko National University of Kyiv, Ukraine; e-mail address: oksanabraganets@knu.ua} \ \ and \ \ Alexander Iksanov\footnote{Faculty of Computer Science and Cybernetics, Taras Shevchenko National University of Kyiv, Ukraine; e-mail address: iksan@univ.kiev.ua}}
\maketitle

\begin{abstract}
\noindent
We investigate a nested balls-in-boxes scheme in a random environment. The boxes follow a nested hierarchy, with infinitely many boxes in each level, and the hitting probabilities of boxes are
random and obtained by iterated fragmentation of a unit mass. The hitting probabilities of the first-level boxes are given by a stick-breaking model $P_k = W_1 W_2\cdot \ldots\cdot W_{k-1}(1- W_k)$ for $k \in \mathbb{N}$, where $W_1$, $W_2,\ldots$ are independent copies of a random variable $W$ taking values in $(0,1)$. The infinite balls-in-boxes scheme in the first level is known as a Bernoulli sieve. We assume that the mean of $|\log W|$ is infinite and the distribution tail of $|\log W|$ is regularly varying at $\infty$. Denote by $K_n(j)$ the number of occupied boxes in the $j$th level provided that there are $n$ balls and call the level $j$ intermediate, if $j = j_n \to \infty$ and $j_n = o((\log n)^a)$ as $n \to \infty$ for appropriate $a>0$. We prove that, for some intermediate levels $j$, finite-dimensional distributions of the process $(K_n(\lfloor j_n u\rfloor))_{u>0}$, properly normalized, converge weakly as $n\to\infty$ to those of a pathwise Lebesgue-Stieltjes integral, with the integrand being
an exponential function and the integrator being an inverse stable subordinator.
The present paper continues the line of investigation
initiated in the articles Buraczewski, Dovgay and Iksanov (2020) and Iksanov, Marynych and Samoilenko (2022) in which the random variable $|\log W|$ has a finite second moment, and Iksanov, Marynych and Rashytov (2022) in which $|\log W|$ has a finite mean and an infinite second moment.
\end{abstract}



\section{Introduction}

\subsection{Definition of the model}

Let $P_1$, $P_2,\ldots$ be nonnegative random variables with an arbitrary joint distribution satisfying $\sum_{k\geq 1}P_k=1$ almost surely (a.s.). The sequence $(P_k)_{k\geq 1}$ is interpreted as a {\it random environment}.
An {\it occupancy scheme in random environment} is defined as follows. Conditionally on $(P_k)_{k\geq 1}$, balls are allocated independently over an infinite array of boxes $1$, $2,\ldots$ with probability $P_k$ of hitting box $k$. The occupancy scheme in random environment is called {\it infinite}, if the number of positive probabilities is infinite a.s. One of the most popular infinite occupancy schemes in random environment is called {\it Bernoulli sieve}. It corresponds to $(P_k)_{k\geq 1}$ which follow a stick-breaking model
\begin{equation}\label{eq:stick}
P_k=W_1W_2\cdot\ldots\cdot W_{k-1}(1-W_k),\quad k\in\mn:=\{1,2,\ldots\},
\end{equation}
where $W_1$, $W_2,\ldots$ are independent copies of a random variable $W$ taking values in $(0,1)$. The Bernoulli sieve was introduced in \cite{Gnedin:2004} and further investigated in many articles. Surveys of earlier works can be found in \cite{Gnedin+Iksanov+Marynych:2010, Iksanov:2016}. An incomplete list of more recent contributions includes \cite{Alsmeyer+Iksanov+Marynych:2017, Duchamps+Pitman+Tang:2019, Iksanov+Jedidi+Bouzzefour:2017, Pitman+Tang:2019, Pitman+Yakubovich:2017, Pitman+Yakubovich:2019}.

A {\it nested infinite occupancy scheme in random environment} is a hierarchical generalization of the infinite occupancy scheme in random environment which is defined by settling, in a consistent way, the sequence
of the occupancy schemes in random environment on the tree of a 
weighted branching process with positive weights (a multiplicative counterpart of a branching random walk). The scheme was introduced in \cite{Bertoin:2008} and further investigated in \cite{Braganets+Iksanov:2023, Buraczewski+Dovgay+Iksanov:2020, Gnedin+Iksanov:2020, Iksanov+Mallein:2022, Iksanov+Marynych+Samoilenko:2022, Joseph:2011}. Following \cite{Bertoin:2008, Gnedin+Iksanov:2020}, we now provide the definition of the nested occupancy scheme in random environment. To this end, some preparations are needed. Let $\mathbb{V}:=\bigcup_{n\geq 0}\mn^n$ be the set of all possible individuals of some population. The initial ancestor is identified with the empty word $\oslash$ and its weight is $P(\oslash)=1$. On some probability space, let $((P_k(v))_{k\geq 1})_{v\in\mathbb{V}}$ denote a family of independent copies of $(P_k)_{k\geq 1}$. An individual $v=v_1\ldots v_j$ in the $j$th level (generation) whose weight is denoted by $P(v)$ produces an infinite number of offspring residing in the $(j+1)$st level. The offspring of the individual $v$ are enumerated by $vk=v_1\ldots v_j k$, where $k\in\mn$, and the weights of the offspring are denoted by $P(vk)$. It is postulated that $P(vk)=P(v)P_k(v)$. Note that, for each $j\in\mn$, $\sum_{|v|=j}P(v)=1$ a.s., where, by convention, $\sum_{|v|=j}$ means that the sum is taken over all individuals in the $j$th level.

We are ready to explain the construction of the nested occupancy scheme in random environment. We identify individuals with boxes, so that the random weights (probabilities) of boxes in the subsequent levels are formed by the vectors $(P(v))_{|v|=1}=(P_k)_{k\geq 1}$, $(P(v))_{|v|=2}$ and so on. The collection of balls is the same for all levels. The balls are allocated, conditionally on $(P(v))_{v\in\mathbb{V}}$, according to the following rule.
At time $0$, infinitely many balls are collected in the box $\oslash$. At the time $n\in\mn$, one of the balls leaves the root $\oslash$ and falls, independently of the $(n-1)$ balls that have left the root earlier, into the box $v$ in the first level with probability $P(v)$. Simultaneously, it falls into the box $vi_1$ in the second level with probability $P(vi_1)/P(v)$, into the box $vi_1i_2$ in the third level with probability $P(vi_1i_2)/P(vi_1)$ and so on, indefinitely. At time $n$, that is, when $n$ balls have already spread over the tree, a box is deemed occupied provided it was hit by some ball (out of $n$) on its way over the levels.

In what follows, we assume that the probabilities $(P_k)_{k\geq 1}$ are given by \eqref{eq:stick} and that these and the outcome of throwing balls are defined on a common probability space. Let $K_n(j)$ be the number of occupied boxes at time $n$ in the $j$th level. Assuming that the mean of $|\log W| $ is infinite and the distribution tail of $|\log W|$ is regularly varying at $\infty$, we investigate the distributional behavior of $K_n(j)$ when $j=j_n\to\infty$ and $j_n=o((\log n)^a)$ as $n \to \infty$ for appropriate $a>0$. We call {\it intermediate} the levels $j$ satisfying the latter two properties. The analogous problem was studied in \cite{Buraczewski+Dovgay+Iksanov:2020, Iksanov+Marynych+Samoilenko:2022} in the situation in which the variable $|\log W|$ has a finite second moment, and in \cite{Iksanov+Marynych+Rashytov:2022} in the situation in which  $|\log W|$ has a finite mean and an infinite second moment.

By Theorem 1 in \cite{Joseph:2011}, the height of the scheme defined by $\tau_n:=\inf\{j\geq 1: K_n(j)=n\}$ is of order $\log n$, irrespective of the distribution tail of $|\log W|$. The asymptotic behavior of the scheme in the levels close to the height (so called {\it late} levels) was investigated in \cite{Bertoin:2008} and \cite{Iksanov+Mallein:2022}. The results obtained in the cited articles apply both in the heavy-tailed scenario considered here and in a light-tailed scenario.

\subsection{Our assumptions on the distribution of $W$}\label{sect:assump}

Our main result, Theorem \ref{main}, will be stated under certain assumptions on the distribution of $(|\log W|, |\log (1-W)|)$. We think it is instructive to introduce the assumptions for an arbitrary random vector $(\xi, \eta)$ and then specialize for $(\xi, \eta)=(|\log W|, |\log (1-W)|)$.

Let $(\xi_k, \eta_k)_{k\geq 1}$ be independent copies of a random vector $(\xi, \eta)$ with positive arbitrary dependent components. Let $F$ and $G$ be the distribution functions of $\xi$ and $\eta$, respectively. 
Denote by $(S_k)_{k\geq 0}$ the zero-delayed standard random walk with increments $\xi_k$, that is, $S_0: = 0$ and $S_k:= \xi_1+ \ldots + \xi_k$ for $k \in \mathbb{N}$. Put
\begin{equation*}
T_k := S_{k-1} + \eta_k,\quad k\in\mn.
\end{equation*}
The random sequence $T = (T_k)_{k\geq 1}$ is known in the literature as a (globally) perturbed random walk, see \cite{Iksanov:2016} for a survey  and  \cite{Basraketal:2022, Bohdanskyietal:2024, Bohunetal:2022, Braganets+Iksanov:2023, jap2023} for more recent contributions. A connection with the occupancy scheme treated in the paper is justified by the fact that if $(P_k)_{k\geq 1}$ is given by \eqref{eq:stick}, then $(-\log P_k)_{k\geq 1}$ is a perturbed random walk with $(\xi,\eta)=(|\log W|, |\log(1-W)|)$.

Put $N(t) := \sum_{k\geq 1} \1_{\{T_k\leq t\}}$ and $V(t) := \me N(t)$ for $t\geq 0$. Plainly,
\begin{equation}\label{eq:V}
V(t)= \int_{[0,\,t]}G(t-y) {\rm d}U(y), \quad t\geq 0,
\end{equation}
where, for $t\geq 0$, $U(t) = \sum_{i\geq 0} \mmp\{S_i \leq t\}$ is the renewal function.

Assume that
\begin{equation}\label{eq:standing assumption}
\mmp\{\xi>t\}\sim ct^{-\alpha},\quad t\to\infty
\end{equation}
for some $c>0$ and $\alpha\in (0,1)$. A standard result of renewal theory states that the latter limit relation is equivalent to $$U(t)~\sim~ Ct^\alpha,\quad t\to\infty,$$ where $C=C_\alpha:=(c\Gamma(1+\alpha)\Gamma(1-\alpha))^{-1}$ and $\Gamma$ is the Euler gamma-function. To prove this, one uses the equality $\int_{[0,\,\infty)}\eee^{-st}{\rm d}U(t)=(1-\varphi(s))^{-1}$ for $s>0$, where $\varphi(s):=\me \eee^{-s\xi}$ for $s\geq 0$, and Karamata's Tauberian theorem (Theorem 1.7.1 on p.~37 in \cite{bingham87}
). Since $\int_{[0,\,\infty)}\eee^{-st}{\rm d}V(t)=\me \eee^{-s\eta}(1-\me \eee^{-s\xi})^{-1}$ for $s>0$, the same reasoning enables us to conclude that \eqref{eq:standing assumption} is equivalent to $V(t)\sim Ct^\alpha$ as $t\to\infty$. It will become clear soon that the first-order behavior of $V$ is not sufficient for the purposes of the present work. A two-term expansion of $V$ is needed. It will be shown in Section \ref{sect:discussion} that the following two assumptions ensure the required expansion of $V$. 

\noindent {\sc Assumption $A$:} $U(t)=Ct^\alpha+O(t^\rho)$ for some $\rho\in [0,\alpha)$.

\noindent {\sc Assumption $B$:} Either ${\lim\sup}_{t\to\infty}(1-G(t))/(1-F(t))=m_0\in [0,\infty)$ or $m_0=\infty$ and that ${\lim\sup}_{t\to\infty} t^\theta (1-G(t))=m_1\in [0,\infty)$ necessarily for some $\theta\in (0,\alpha)$.

It seems to be a non-trivial problem to provide necessary and sufficient conditions for Assumption $A$ formulated in terms of the distribution of $\xi$. In the case where the distribution of $\xi$ is absolutely continuous (with respect to Lebesgue measure) sufficient conditions in terms of the densities considered as functions of a complex argument can be found in \cite{Wong:1976}. In Section \ref{sect:discussion} we give two sets of sufficient conditions, each of a different flavor.

Observe that when applied to the vector $(\xi,\eta)=(|\log W|, |\log (1-W)|)$ Assumptions $A$ and $B$ regulate the behavior of $W$ near $0$ and near $1$, respectively. If (a) $\me \eee^{-s|\log W|}=\exp (-c\Gamma(1-\alpha)s^\alpha)$ for $s\geq 0$, that is, the distribution of $|\log W|$ is $\alpha$-stable or (b) $\me \eee^{-s|\log W|}=(1+(c/\kappa)\Gamma(1-\alpha)s^\alpha)^{-\kappa}$ for $s\geq 0$ and some $\kappa>0$, then Assumption $A$ holds with $\rho=0$ and Assumption $B$ holds with $m_0=0$. This statement will be proved in Section \ref{sect:discussion}.

\section{Main result}\label{sect:mainresult}

To formulate our main result, we recall the definition of an {\it inverse stable subordinator}.
Let $\mathcal{Z}_\alpha=(\mathcal{Z}_\alpha(u))_{u\geq 0}$ be an $\alpha$-stable subordinator with
$-\log \me \eee^{-s\mathcal{Z}_\alpha(1)}=\Gamma(1-\alpha)s^\alpha$ for $s\geq 0$. The process $\mathcal{Z}_\alpha^\leftarrow=(\mathcal{Z}_\alpha^\leftarrow(u))_{u\geq 0}$ defined by $$\mathcal{Z}_\alpha^\leftarrow(u):=\inf\{v\geq 0: \mathcal{Z}_\alpha(v)>u\},\quad u\geq 0$$ is called {\it inverse $\alpha$-stable subordinator.}

We write $\overset{{\rm f.d.d.}}{\longrightarrow} $ to denote weak convergence of finite-dimensional distributions. Also, $\lfloor x\rfloor$ denotes the integer part of $x\in\mr$.
\begin{theorem}\label{main}
Under Assumptions $A$ and $B$ imposed on $(\xi, \eta)=(|\log W|, |\log (1-W)|)$, let $(j_n)_{n\geq 1}$ be a sequence of positive numbers satisfying $\lim_{n\to\infty}j_n=\infty$ and
\begin{equation*}
j_n = o\left((\log n)^{\min \left(1/3,(\alpha-\beta)/(\alpha-\beta+1)\right)}\right),\quad n\to\infty,
\end{equation*}
where $\beta=\rho$ if $m_0<\infty$ and $\beta=\max (\rho, \alpha-\theta)$ if $m_0=\infty$ and $m_1<\infty$. Then
$$\left(\frac{c j_n ^\alpha }{\rho_{\lfloor j_n u \rfloor - 1}(\log n)^{\alpha \lfloor j_n u \rfloor}} K_n(\lfloor j_n u \rfloor )
\right)_{u>0}~\overset{{\rm f.d.d.}}{\longrightarrow}~\left(\int_0^\infty \eee^{-\alpha u y} {\rm d} \mathcal{Z}_\alpha^\leftarrow(y) \right)_{u>0},\quad n\to\infty,
$$
where
\begin{equation}\label{def:fi}
\rho_i:= \frac{(C\Gamma(\alpha + 1))^i}{\Gamma(\alpha i + 1)}, \quad i\in \mn_0:=\{0,1,\ldots\},
\end{equation}
and $\Gamma$ is the Euler gamma-function.
\end{theorem}

Now we discuss the structure of the limit process in Theorem \ref{main}. It can be seen from the proof that the integrator $\mathcal{Z}_\alpha^\leftarrow$ describes the fluctuations of the number of occupied boxes in the first level, properly normalized. The process $\mathcal{Z}_\alpha^\leftarrow$ is a.s.\ continuous singular with respect to Lebesgue measure. The integrand $y\mapsto \eee^{-\alpha u y}$ which is formed by the renewal structure of the underlying tree makes the process $\big(\int_0^\infty \eee^{-\alpha u y} {\rm d} \mathcal{Z}_\alpha^\leftarrow(y)\big)_{u>0}$ a.s.\ infinitely differentiable. Thus, when passing from the first level to intermediate levels the fluctuations of the number of occupied boxes smooth out.

Since $\mathcal{Z}_\alpha^\leftarrow(\mathcal{Z}_\alpha(y))=y$ a.s., we conclude that $$\Big(\int_0^\infty \eee^{-\alpha u y} {\rm d} \mathcal{Z}_\alpha^\leftarrow(y)\Big)_{u>0}=\Big(\int_0^\infty \eee^{-\alpha u \mathcal{Z}_\alpha(y)} {\rm d}y\Big)_{u>0}\quad\text{a.s.}$$ Further, by the scaling property of $\mathcal{Z}_\alpha$, with $u>0$ fixed, $$\int_0^\infty \eee^{-\alpha u \mathcal{Z}_\alpha(y)} {\rm d}y\overset{{\rm d}}{=} (\alpha u)^{-\alpha}\int_0^\infty \eee^{-\mathcal{Z}_\alpha(y)}{\rm d}y=: (\alpha u)^{-\alpha} I_\alpha,$$ where $\overset{{\rm d}}{=}$ denotes equality of distributions. The latter integral is known in the literature as an exponential functional of the subordinator $\mathcal{Z}_\alpha$. According to Proposition 3.3 in \cite{Carmona+Petit+Yor:1997}, $$\me (I_\alpha)^n 
=\frac{(n!)^{1-\alpha}}{(\Gamma(1-\alpha))^n},\quad n\in\mn$$ and furthermore $\me \exp(sI_\alpha 
)<\infty$ for all $s>0$. It can be shown (see Proposition 3.4(iv) in \cite{Carmona+Petit+Yor:1997}) that the moment formula implies that $\log I_\alpha + \log \Gamma(1-\alpha) \overset{{\rm d}}{=} X(1-\alpha)$, where $(X(u))_{u\geq 0}$ is a L\'{e}vy process with $X(1)\overset{{\rm d}}{=} \log \mathcal{E}$, and $\mathcal{E}$ is a random variable with the exponential distribution of unit mean. A consequence of this observation is that the distribution of $I_\alpha$ is absolutely continuous with respect to Lebesgue measure.

Next, we intend to set a non-rigorous link between Theorem \ref{main} and a limit theorem for $K_n(j)$ with fixed $j$. Assume that $\mmp\{|\log W|>x\}\sim x^{-\alpha}\ell(x)$ as $x\to\infty$ for some $\alpha\in (0,1)$ and some $\ell$ slowly varying at $\infty$. A specialization of a functional limit theorem given in Theorem 7 of \cite{Braganets+Iksanov:2023} reads $$\Big(\frac{(\ell(\log n))^j K_n(j)}{(\log n)^{\alpha j}}\Big)_{j\geq 1}~\overset{{\rm d}}{\longrightarrow}~\Big(\int_{[0,\,1]}(1-y)^{\alpha(j-1)}{\rm d} \mathcal{Z}_\alpha^\leftarrow(y)\Big)_{j\geq 1},\quad n\to\infty,$$ where $\overset{{\rm d}}{\longrightarrow}$ denotes convergence in distribution. Observe that $(\mathcal{Z}_\alpha^\leftarrow(y/j))_{y\geq 0}$ has the same distribution as $(j^{-\alpha}\mathcal{Z}_\alpha^\leftarrow(y))_{y\geq 0}$ and that, with $y>0$ fixed, $\lim_{j\to\infty}(1-y/j)^{\alpha (j-1)}=\eee^{-\alpha y}$. Hence, $$j^\alpha \int_{[0,\,1]}(1-y)^{\alpha(j-1)}{\rm d} \mathcal{Z}_\alpha^\leftarrow(y)=j^\alpha \int_{[0,\,j]}(1-y/j)^{\alpha(j-1)}{\rm d} \mathcal{Z}_\alpha^\leftarrow(y/j)\overset{{\rm d}}{=} \int_{[0,\,j]}(1-y/j)^{\alpha(j-1)}{\rm d} \mathcal{Z}_\alpha^\leftarrow(y).$$ 
It is reasonable to expect that the distributional limit of the last integral as $j\to\infty$ is $\int_{[0,\,\infty)}\eee^{-\alpha y}{\rm d} \mathcal{Z}_\alpha^\leftarrow(y)$. This is in line with Theorem \ref{main}.

The remainder of the article is structured as follows. As has already been announced, we discuss in Section \ref{sect:discussion} Assumptions $A$ and $B$ and their consequences. Section \ref{sect:perturbed} contains auxiliary results on a general branching process generated by a perturbed random walk. The proof of Theorem \ref{main} is given in Section \ref{sect:proofs}.

\section{Discussion of Assumptions $A$ and $B$}\label{sect:discussion}

We start by showing that Assumptions $A$ and $B$ secure a two-term expansion of $V$.

\begin{lemma}\label{lem:estim1}
Assume that Assumptions $A$ and $B$ hold. Then there exists $D>0$ such that
\begin{equation}\label{eq:estim}
|V(t)-Ct^\alpha|\leq Dt^\beta,\quad t\geq 0,
\end{equation}
where $\beta=\rho$ if $m_0<\infty$, $\beta=\max(\rho,\alpha-\theta)$ if $m_0=\infty$ and $m_1<\infty$; $C>0$, $\alpha\in (0,1)$ and $\rho\in [0,\alpha)$ are as defined in Assumption $A$; $m_0$, $m_1$ and $\theta\in (0,\alpha)$ are as defined in Assumption $B$.
\end{lemma}
\begin{proof}
Assuming that ${\lim\sup}_{t\to\infty}(1-G(t)/(1-F(t))=m_0\in [0,\infty)$ we shall show that
\begin{equation}\label{eq:inter2}
\int_{[0,\,t]}(1-G(t-y)){\rm d}U(y)=O(1),\quad t\to\infty.
\end{equation}
Indeed, given $m_2>m_0$ there exists $t_0>0$ such that $1-G(t)\leq m_2(1-F(t))$ whenever $t\geq t_0$. With this at hand, $$\int_{[0,\,t-t_0]}(1-G(t-y)){\rm d}U(y)\leq m_2 \int_{[0,\,t]}(1-F(t-y)){\rm d}U(y)=m_2.$$ The last equality is another form of writing the renewal equation for $U$. Alternatively, it is equivalent to $\mmp\{\inf\{k\geq 0: S_k>t\}<\infty\}=1$. By subadditivity of $U$ (see, for instance, Theorem 1.7 on p.~10 in \cite{Mitov+Omey:2014}), $$\int_{(t-t_0,\,t]}(1-G(t-y)){\rm d}U(y)\leq U(t)-U(t-t_0)\leq U(t_0),$$ and \eqref{eq:inter2} follows. Equality \eqref{eq:V} is equivalent to $$V(t)=U(t)-\int_{[0,\,t]}(1-G(t-y)){\rm d}U(y),\quad t\geq 0.$$ Hence, Assumption $A$ entails $V(t)=Ct^\alpha+O(t^\rho)$ as $t\to\infty$. Since $V(0)=0$, we can find $D>0$ such that $|V(t)-Ct^\alpha|\leq Dt^\rho$ for all $t\geq 0$.

Assume now that $m_0=\infty$ and that ${\lim\sup}_{t\to\infty} t^\theta (1-G(t))=m_1\in [0,\infty)$ for some $\theta\in (0,\alpha)$. We shall prove that
\begin{equation}\label{eq:inter3}
\int_{[0,\,t]}(1-G(t-y)){\rm d}U(y)=O(t^{\alpha-\theta}),\quad t\to\infty.
\end{equation}
Given $m_3>m_1$ and given $ C^\ast> C$ there exists $t_1$ such that $U(t)\leq C^\ast t^\alpha$ and $1-G(t)\leq m_3 t^{-\theta}$ whenever $t\geq t_1$. We already know from the previous paragraph that subadditivity of $U$ ensures $\int_{(t-t_1,\,t]}(1-G(t-y)){\rm d}U(y)=O(1)$ as $t\to\infty$. Further, we obtain with the help of integration by parts,
\begin{multline*}
\int_{[0,\,t-t_1]}(1-G(t-y)){\rm d}U(y)\leq m_3 \int_{[0,\,t-t_1]}(t-y)^{-\theta}{\rm d}U(y)= m_3
\int_{[t_1,\,t]}y^{-\theta}{\rm d}_y(U(t)-U(t-y))\\= m_3 \Big(  t^{-\theta}(U(t)-U(0))-t_1^{-\theta}(U(t)-U(t-t_1))+\theta \int_{t_1}^t (U(t)-U(t-y))y^{-\theta-1}{\rm d}y  \Big).
\end{multline*}
Here and elsewhere, we write ${\rm d}_y$ rather than ${\rm d}$ whenever there is an ambiguity, that is, a function under the differential depends on $y$ and some other variables. The first term is $O(t^{\alpha-\theta})$ and the second term is $O(1)$. Another appeal to subadditivity of $U$ yields $$\int_{t_1}^t (U(t)-U(t-y))y^{-\theta-1}{\rm d}y\leq \int_{t_1}^t U(y)y^{-\theta-1}{\rm d}y\leq C^\ast \int_{t_1}^t y^{\alpha-\theta-1}{\rm d}y=O(t^{\alpha-\theta}),$$ and \eqref{eq:inter3} follows. Mimicking the argument from the first part of the proof we arrive at \eqref{eq:estim} with $\beta=\max(\rho,\alpha-\theta)$.
\end{proof}

Let $(\mathcal{Z}_\alpha(u))_{u\geq 0}$ be an $\alpha$-stable subordinator as defined at the beginning of Section \ref{sect:mainresult}. Now we discuss two sets of sufficient conditions for Assumption A. The first of these was outlined by one of the referees. Observe that $$\Phi_\alpha(v):=\me \eee^{{\rm i}v\mathcal{Z}_\alpha(1)}=\exp(-\Gamma(1-\alpha)|v|^\alpha(\cos(\pi\alpha/2)-{\rm i}\sin(\pi\alpha/2){\rm sgn}\,v),\quad v\in\mr.$$
\begin{lemma}\label{lem:suff}
Assume that there exist constants $a_1, a_2>0$ and $r\in (\alpha, 1]$ such that, for all $v\in\mr$ satisfying $|v|\leq a_1$,
\begin{equation}\label{eq:ineq}
|\me \eee^{{\rm i}v\xi}-\Phi_\alpha(v)|\leq a_2 |v|^r.
\end{equation}
Then Assumption A holds with $\rho=0$ if $r>2\alpha$, $\rho=\delta$ for any $\delta>0$ if $r=2\alpha$ and $\rho=\lambda(2\alpha-r)$ for any $\lambda>1$ such that $\lambda (2\alpha-r)<\alpha$ if $r<2\alpha$.
\end{lemma}
\begin{proof}
Observe that $\int_0^\infty \mmp\{\mathcal{Z}_\alpha(cu)\leq t\}{\rm d}u=Ct^\alpha$ for $t\geq 0$ and that $$0\leq \sum_{n\geq 0}\mmp\{\mathcal{Z}_\alpha(cn)\leq t\}-\int_0^\infty \mmp\{\mathcal{Z}_\alpha(cu)\leq t\}{\rm d}u\leq 1,\quad t\geq 0.$$ Hence, it is enough to prove that
\begin{equation}\label{eq:aux}
\Big|\sum_{n\geq 0}\mmp\{S_n\leq t\}-\sum_{n\geq 0}\mmp\{\mathcal{Z}_\alpha(cn)\leq t\}\Big|=O(t^\rho),\quad t\to\infty.
\end{equation}
By Proposition 1 in \cite{Paulauskas:1974},
\begin{equation}\label{eq:paul}
\sup_{x\geq 0}\big|\mmp\{S_n\leq n^{1/\alpha}x\}-\mmp\{\mathcal{Z}_\alpha(c)\leq x\}\big|=\sup_{t\geq 0}\big|\mmp\{S_n\leq t\}-\mmp\{\mathcal{Z}_\alpha(cn)\leq t\}\big|\leq Bn^{-\alpha^{-1}(r-\alpha)}
\end{equation}
for each $n\in\mn$ and a constant $B>0$ which does not depend on $n$.

\noindent {\sc Case $r\in (2\alpha, 1)$} (necessarily $\alpha<1/2$). According to \eqref{eq:paul}, relation \eqref{eq:aux} holds with $\rho=0$. 

\noindent {\sc Case $r\in (\alpha, 2\alpha)$}, in which $0<\alpha^{-1}(r-\alpha)<1$. Fix any $\lambda>1$ such that $\lambda(2\alpha-r)<\alpha$. Using \eqref{eq:paul} we infer $$\Big|\sum_{n=0}^{\lfloor t^{\lambda\alpha}\rfloor} \mmp\{S_n\leq t\}-\sum_{n=0}^{\lfloor t^{\lambda\alpha} \rfloor}\mmp\{\mathcal{Z}_\alpha(cn)\leq t\}\Big|\leq B\sum_{n=1}^{\lfloor t^{\lambda\alpha}\rfloor}n^{-\alpha^{-1}(r-\alpha)}=O(t^{\lambda(2\alpha-r)}),\quad t\to\infty.$$ We claim that
\begin{equation}\label{eq:inter4}
\sum_{n\geq \lfloor t^{\lambda\alpha}\rfloor+1}\mmp\{S_n\leq t\}=o(1)\quad\text{and}\quad \sum_{n\geq \lfloor t^{\lambda\alpha}\rfloor+1}\mmp\{\mathcal{Z}_\alpha(cn)\leq t\}=o(1),\quad t\to\infty.
\end{equation}
Indeed, by Markov's inequality, for each $u>0$, $$\sum_{n\geq \lfloor t^{\lambda\alpha}\rfloor+1}\mmp\{S_n\leq t\}\leq \eee^{ut}\sum_{n\geq \lfloor t^{\lambda\alpha}\rfloor+1}(\varphi(u))^n\leq \frac{\eee^{ut+\log \varphi(u)t^{\lambda\alpha}}}{1-\varphi(u)}.$$ Put $u=1/t$. Then $(1-\varphi(1/t))^{-1}\sim (c\Gamma(1-\alpha))^{-1}t^\alpha$ as $t\to+\infty$. Since $-t^{\lambda\alpha}\log \varphi(1/t)\sim c\Gamma(1-\alpha)t^{(\lambda-1)\alpha}$ as $t\to+\infty$, the numerator decays to $0$ superexponentially fast. This proves the first limit relation in \eqref{eq:inter4}. The argument for the second limit relation is analogous.

\noindent {\sc Case $r=2\alpha<1$}. In the proof for the previous case, pick any $\lambda>1$. Then $$\Big|\sum_{n=0}^{\lfloor t^{\lambda\alpha}\rfloor} \mmp\{S_n\leq t\}-\sum_{n=0}^{\lfloor t^{\lambda\alpha} \rfloor}\mmp\{\mathcal{Z}_\alpha(cn)\leq t\}\Big|=O(\log t),\quad t\to\infty.$$ Also, relations \eqref{eq:inter4} hold true in the present setting.
\end{proof}

We proceed by discussing a possible application of Lemma \ref{lem:suff}. Assume that $\xi\overset{{\rm d}}{=} \mathcal{Z}_\alpha(c)+\theta$ for a nonnegative random variable $\theta$ which is independent of $\mathcal{Z}_\alpha(c)$. If $\me\theta<\infty$, then inequality \eqref{eq:ineq} holds with $r=1$. Further, if $\alpha=3/4$, say, then Assumption A holds with $\rho=\lambda/2$ for any $\rho\in (1,3/2)$. If $\mmp\{\theta>t\}\sim {\rm const}\,t^{-\beta}$ as $t\to\infty$ for some $\beta\in (\alpha,1)$, then inequality \eqref{eq:ineq} holds with $r=\beta$. Further, if $\alpha=1/4$ and $\beta=3/4$, then Assumption A holds with $\rho=0$.

Next, 
we give other sufficient conditions for Assumption $A$ which works for a special class of the distributions of $\xi$. Namely, assume that the distribution of $\xi$ coincides with that of $\mathcal{Z}_\alpha(c) X^{1/\alpha}$, where $X$ is a positive random variable of mean one which is independent of $\mathcal{Z}_\alpha(c)$ for some $c>0$. Equivalently, this assumption means that the standard random walk $(S_n)_{n\geq 0}$ has the same distribution as $(\mathcal{Z}_\alpha(c\hat S_n))_{n\geq 0}$. Here, $\hat S_0:=0$, $\hat S_n:=X_1+\ldots+X_n$ for $n\in\mn$, $X_1$, $X_2,\ldots$ are independent copies of $X$, and $(\hat S_n)_{n\geq 0}$ is independent of $(\mathcal{Z}_\alpha(u))_{u\geq 0}$.

Put $\hat U(t):=\sum_{n\geq 0}\mmp\{\hat S_n\leq t\}$ for $t\geq 0$. Now we obtain a formula which connects $U$ and $\hat U$. To this end, we need the following property: for fixed $t>0$, $\mathcal{Z}_\alpha(c t)$ has the same distribution as $t^{1/\alpha} \mathcal{Z}_\alpha(c)$. As a consequence,
\begin{multline}\label{u_equation}
U(t)  = \sum_{n \geq 0} \mmp \{ S_n \leq t\} = \sum_{n \geq 0} \mmp \{ \mathcal{Z}_\alpha (c\hat S_n) \leq t\} 
=  \sum_{n \geq 0} 
\mmp  \{\hat S_n^{1/\alpha} \mathcal{Z}_\alpha (c) \leq t 
\}
\\= \sum_{n \geq 0} 
\mmp  \{ \hat S_n \leq  (\mathcal{Z}_\alpha(c))^{-\alpha}  t^\alpha 
\} 
= \me \hat U\left( (\mathcal{Z}_\alpha(c))^{-\alpha}  t^\alpha \right),\quad t\geq 0.
\end{multline}

Let $d>0$. Recall that the distribution of a positive random variable is called {\it $d$-arithmetic} if it is concentrated on the lattice $(nd)_{n\geq 1}$ and not concentrated on $(nd_1)_{n\geq 1}$ for any $d_1>d$. The distribution of $Y$ is called {\it nonarithmetic} if it is not $d$-arithmetic for all $d>0$.

Invoking \eqref{u_equation} enables us to provide sufficient conditions which ensure Assumption $A$.
\begin{lemma}\label{lem:assump}
Assume that $\me X^\vartheta<\infty$ for some $\vartheta\in (1,2]$. Then $$U(t)=Ct^\alpha+O(t^{\alpha(2-\theta)}),\quad t\to\infty,$$
that is, Assumption $A$ holds with $\rho=\alpha(2-\vartheta)$. Here, as before, $C=(c \Gamma(1+\alpha) \Gamma(1-\alpha))^{-1}$.
\end{lemma}
\begin{proof}
We intend to use Lorden's inequality for $\hat U$. Although, in many sources it is only stated for nonarithmetic distributions, it holds and takes the same form for both nonarithmetic and arithmetic distributions, see \cite{Carlsson+Nerman:1986} for an elegant proof.

If $\me X^2<\infty$, then, by Lorden's inequality, $\hat U(t)\leq t+\me X^2$ for all $t\geq 0$ (recall that $\me X=1$). Hence,
$$
U(t) = \me \hat U\left( (\mathcal{Z}_\alpha(c))^{-\alpha}  t^\alpha \right) \leq \me  (\mathcal{Z}_\alpha(c))^{-\alpha}  t^\alpha  + \me X^2.
$$
It remains to note that, by Lemma \ref{lem:formula},
\begin{equation*}
\me(\mathcal{Z}_\alpha(c))^{-\alpha} = \frac{\alpha}{\Gamma(1+ \alpha)} \int_0^\infty s^{\alpha-1} \eee^{-c\Gamma(1-\alpha)s^\alpha} {\rm d} s=C.
\end{equation*}

If $\me X^\vartheta<\infty$ for some $\vartheta\in (1,2)$, then $\hat U(t)=t+O(t^{2-\vartheta})$ as $t\to\infty$, see, for instance, formula (3.13) on p.~433 in \cite{Bohunetal:2022}. In particular, for some $t_0>0$ and a constant $C^\ast>0$ $\hat U(t) \leq t+ C^\ast t^{2-\vartheta}$ whenever $t\geq t_0$. Using \eqref{u_equation} yields
\begin{multline*}
U(t)\leq\me \big( (\mathcal{Z}_\alpha(c))^{-\alpha}  t^\alpha  + C^\ast (\mathcal{Z}_\alpha(c))^{-\alpha(2-\vartheta)} t^{\alpha(2-\vartheta)} \big)\1_{\{(\mathcal{Z}_\alpha(c))^{-\alpha}  t^\alpha\geq t_0\}}\\+\me \hat U((\mathcal{Z}_\alpha(c))^{-\alpha}  t^\alpha)\1_{\{(\mathcal{Z}_\alpha(c))^{-\alpha}  t^\alpha<t_0\}}\leq Ct^\alpha+C^\ast \me (\mathcal{Z}_\alpha(c))^{-\alpha(2-\theta)}t^{\alpha(2-\theta)}+\hat U(t_0)\\=Ct^\alpha+O(t^{\alpha(2-\vartheta)}),\quad t\to\infty.
\end{multline*}

Put $\hat \nu(t):=\inf\{k\geq 1: \hat S_k>t\}$ for $t\geq 0$. Using Wald's identity we infer, for all $t\geq 0$, $\hat U(t)=\me \hat S_{\hat \nu(t)}\geq t$, whence $U(t)\geq Ct^\alpha$. This completes the proof.
\end{proof}

Next, we show that Assumptions $A$ and $B$ hold true, with $(\xi,\eta)=(|\log W|, |\log(1-W)|)$, if, for instance, (a) $|\log W|=\mathcal{Z}_\alpha(c)$ for some $c>0$, that is, the distribution of $|\log W|$ is $\alpha$-stable; (b) $\varphi(s): = \me \eee^{-s\xi}=(1+(c/\kappa)\Gamma(1-\alpha)s^\alpha)^{-\kappa}$ for $s\geq 0$, some $\alpha\in (0,1)$ and some $\kappa>0$. In both cases, $|\log W|$ has the same distribution as $\mathcal{Z}_\alpha(c) X^{1/\alpha}$, where $\mmp\{X=1\}=1$ in the case (a) and $X$ is independent of $\mathcal{Z}_\alpha(c)$ and has a gamma distribution with parameters $\kappa$ and $\kappa$ in the case (b). To justify the latter claim, observe that $s\mapsto \varphi(((c\Gamma(1-\alpha))^{-1}s)^{1/\alpha})=(1+s/\kappa)^{-\kappa}$ is the Laplace-Stieltjes transform of the aforementioned gamma distribution. Thus, in both cases $\me X^2<\infty$ and, according to Lemma \ref{lem:assump}, Assumption $A$ holds with $\rho=0$.

Now, temporarily ignoring the cases (a) and (b), we show that Assumption $B$ holds true, with $m_0=0$, whenever $1-F(x)=\mmp\{|\log W|>x\}\sim cx^{-\alpha}$ as $x\to\infty$ and $\me |\log W|^{-\gamma}<\infty$ for some $\gamma>0$. Indeed, by Markov's inequality,
\begin{multline*} 
1- G(x)=\mmp \{|\log(1-W)|> x \}= \mmp \{|\log W|< |\log(1-\eee^{-x})|\\=\mmp \{|\log W|^{-\gamma} >|\log(1-\eee^{-x})|^{-\gamma}\}\leq \me |\log W|^{-\gamma} (-\log(1-\eee^{-x}))^\gamma~\sim~ \me |\log W|^{-\gamma} \eee^{-\gamma x}
\end{multline*}
as $x \to \infty$, and the claim follows.

We already know that $\me (\mathcal{Z}_\alpha(c))^{-\alpha}<\infty$. Hence, Assumption $B$ holds, with $m_0=0$, in the case (a). Although we do not need such a precision, we note in passing that in the case (a), with $c$ and $\alpha$ satisfying $c\Gamma(1-\alpha)=1$, by Lemma 1 in \cite{hawkes72}, the following asymptotic relation holds $$\mmp \{|\log(1-W)|> x \}~ \sim~ c_1 \eee^{-\frac{\alpha x}{2(1-\alpha)}}\exp(-c_2 |\log(1-\eee^{-x})|^{-\frac{\alpha}{1-\alpha}}),\quad x\to\infty$$ for explicitly known positive constants $c_1$ and $c_2$.

In the case (b), using Lemma \ref{lem:formula} we infer $\me X^{-\gamma}<\infty$ for all $\gamma\in (0,\min (\kappa, 1))$. Thus, again, Assumption $B$ holds, with $m_0=0$. To make the presentation symmetric, we note, without going into details, that Tauberian and Abelian theorems for Laplace transforms can be used to show that
\begin{equation*} 
\mmp \{|\log(1-W)| >x \} 
~\sim~ \frac{\kappa^\kappa}{(c\Gamma(1-\alpha))^\kappa \Gamma(1+\alpha\kappa)} \eee^{-\alpha\kappa x},\quad x\to\infty.
\end{equation*}

\section{An auxiliary general branching process}\label{sect:perturbed}

To prove Theorem \ref{main}, we need some auxiliary results on a general branching process generated by $T$ (a perturbed random walk). Here is its definition in terms of a population of individuals initiated at time $0$ by one individual, the ancestor. An individual born at time $s\geq 0$ produces an offspring whose birth times have the same distribution as $(s+T_k)_{k\geq 1}$. All individuals act independently of each other. An individual resides in the $j$th generation if it has exactly $j$ ancestors.

For $t \geq 0$ and $j \in \mn$, denote by $N_j(t)$ the number of the $j$th generation individuals with birth times smaller than or equal to $t$ and put $V_j(t) := \me N_j(t)$. Then $N_1(t) = N(t)$, and, for $j\geq 2$, $N_j$ admits the following decomposition
$$
N_j(t)=\sum_{r\geq 1}N_{j-1}^{(r)}(t-T_r) \1_{\{T_r \leq t\}}, \quad 
t\geq 0,
$$
where, for $r\in\mn$, $N_{j-1}^{(r)}(t)$ is the number of the $j$th generation individuals who are descendants of the
first-generation individual with birth time $T_r$, and whose birth times fall in $[T_r, T_r+t]$. By the branching property, $(N_{j-1}^{(1)}(t) )_{t\geq 0}$, $(N_{j-1}^{(2)}(t) )_{t\geq 0}, \dots$  are independent copies of $(N_{j-1}(t) )_{t\geq 0}$ which are also independent of $T$. As a consequence, $V_1(t) = V(t)$, and
\begin{equation}\label{Vj:convolution}
V_j(t)  = (V_{j-1}\ast V) (t) = \int_{[0,\,t]} V_{j-1}(t-y) {\rm d} V(y), \quad j\geq 2,~ t\geq 0,
\end{equation}
that is, $V_j=V^{\ast (j)}$ is the $j$-fold Lebesgue-Stieltjes convolution of $V$ with itself. The remainder of this section is concerned with obtaining precise and asymptotic estimates for $V_j$.

According to Lemma \ref{lem:estim1}, Assumptions $A$ and $B$ ensure that inequality \eqref{eq:estim} holds true.
Now we prove that $V_n$ admits a similar estimate.
\begin{lemma}\label{lem:estimate}
Suppose \eqref{eq:estim}. Then
\begin{equation}\label{ineq:first}
|V_n(t)-\rho_n t^{\alpha n}|\leq\sum_{i=0}^{n-1}{n \choose i}\frac{(\Gamma(\alpha+1))^i (\Gamma(\beta+1))^{n-i}}{\Gamma(\alpha i+\beta (n-i)+1)}(Ct^\alpha)^i(Dt^\beta)^{n-i},\quad n\in\mn,~t \geq 0,
\end{equation}
where $\rho_n$ is as given in \eqref{def:fi}.
\end{lemma}
\begin{proof}
We use mathematical induction. For $n=1$, \eqref{ineq:first} reduces to \eqref{eq:estim}. Assume that \eqref{ineq:first} holds with $n-1$ replacing $n$.
Using \eqref{Vj:convolution} for the first equality we conclude that
\begin{multline*}
|V_n(t) - \rho_n t^{\alpha n}|=\Big|\int_{[0,\,t]} V_{n-1}(t-y){\rm d} V(y) - \rho_n t^{\alpha n}\Big|\\\le \int_{[0,\,t]}| V_{n-1}(t-y) - \rho_{n-1} (t-y)^{\alpha(n-1)}| {\rm d} V(y) + \Big|\rho_{n-1}\int_{[0,\, t]}(t-y)^{\alpha(n-1)}{\rm d}V(y)- \rho_n t^{\alpha n}\Big|.
\end{multline*}
Invoking the induction assumption and then integration by parts we obtain
\begin{multline*}
A_n(t):=\int_{[0,\,t]}|V_{n-1}(t-y) - \rho_{n-1} (t-y)^{\alpha(n-1)}| {\rm d} V(y)\\\leq \sum_{i=0}^{n-2}\binom{n-1}{i}\frac{(\Gamma(\alpha+1))^i(\Gamma(\beta+1))^{n-1-i}}{\Gamma(\alpha i+\beta (n-1-i)+1)}D^{n-1-i}C^i \int_{[0,\,t]} (t-y)^{\alpha i+\beta(n-1-i)}{\rm d} V(y)\\=\sum_{i=0}^{n-2} \binom{n-1}{i} \frac{(\Gamma(\alpha+1))^i(\Gamma(\beta+1))^{n-1-i}}{\Gamma(\alpha i+\beta (n-1-i)+1)}D^{n-1-i}C^i \int_0^t V(t-y){\rm d} ( y^{\alpha i+\beta(n-1-i)}).
\end{multline*}
In view of \eqref{eq:estim},
\begin{multline*}
\int_0^t V(t-y){\rm d} ( y^{\alpha i+\beta(n-1-i)})\leq \int_0^t \big(C(t-y)^\alpha+D(t-y)^\beta\big){\rm d}( y^{\alpha i+\beta(n-1-i)})\\=C\frac{\Gamma(\alpha+1)\Gamma(\alpha i+\beta(n-1-i)+1)}{\Gamma(\alpha (i+1)+\beta(n-1-i)+1)} t^{\alpha (i+1)+\beta(n-1-i)}\\+D\frac{\Gamma(\beta+1)\Gamma(\alpha i+\beta(n-1-i)+1)}{\Gamma(\alpha i+\beta(n-i)+1)} t^{\alpha i+\beta(n-i)},
\end{multline*}
whence
\begin{multline}
A_n(t)\leq \sum_{i=0}^{n-2} \binom{n-1}{i} \frac{(\Gamma(\alpha+1))^{i+1}(\Gamma(\beta+1))^{n-1-i}}{\Gamma(\alpha(i+1)+\beta (n-1-i)+1)}D^{n-1-i}C^{i+1}t^{\alpha (i+1)+\beta(n-1-i)}\\+\sum_{i=0}^{n-2} \binom{n-1}{i} \frac{(\Gamma(\alpha+1))^i(\Gamma(\beta+1))^{n-i}}{\Gamma(\alpha i+\beta (n-i)+1)}D^{n-i}C^i t^{\alpha i+\beta(n-i)}\\=(n-1)\frac{(\Gamma(\alpha+1))^{n-1}\Gamma(\beta+1)}{\Gamma(\alpha (n-1)+\beta+1)}D C^{n-1} t^{\alpha (n-1)+\beta}\\+\sum_{i=1}^{n-2} \left(\binom{n-1}{i-1}+\binom{n-1}{i}\right) \frac{(\Gamma(\alpha+1))^i(\Gamma(\beta+1))^{n-i}}{\Gamma(\alpha i+\beta (n-i)+1)}D^{n-i}C^i t^{\alpha i+\beta(n-i)} 
+\frac{(\Gamma(\beta+1))^n}{\Gamma(\beta n+1)}D^nt^{\beta n}\\=(n-1)\frac{(\Gamma(\alpha+1))^{n-1}\Gamma(\beta+1)}{\Gamma(\alpha (n-1)+\beta+1)}D C^{n-1} t^{\alpha (n-1)+\beta}\\+\sum_{i=0}^{n-2} \binom{n}{i}\frac{(\Gamma(\alpha+1))^i(\Gamma(\beta+1))^{n-i}}{\Gamma(\alpha i+\beta (n-i)+1)}D^{n-i}C^i t^{\alpha i+\beta(n-i)}. \label{eq:inter1}
\end{multline}
To obtain the first equality we have changed the variable $i\to i-1$ under the first sum appearing in the first inequality, singled out the term which corresponds to $i=n-1$ and also singled out the term of the second sum which corresponds to $i=0$.

Further, we infer
\begin{multline*}
\Big|\rho_{n-1}\int_{[0,\, t]}(t-y)^{\alpha(n-1)}{\rm d}V(y)- \rho_n t^{\alpha n}\Big|\\=\rho_{n-1}\Big|\int_{[0,\, t]} (t-y)^{\alpha(n-1)}{\rm d}V(y)-C \int_{[0,\, t]} (t-y)^\alpha {\rm d}(y^{\alpha(n-1)})\Big|\\=\Big|\rho_{n-1}\int_0^t \big(V(t-y)-C(t-y)^\alpha\big){\rm d}(y^{\alpha(n-1)})\Big|\leq D\rho_{n-1}\int_0^t (t-y)^\beta {\rm d}(y^{\alpha(n-1)})\\=D\frac{(C\Gamma(\alpha+1))^{n-1}}{\Gamma(\alpha(n-1)+1)}\frac{\Gamma(\alpha(n-1)+1)\Gamma(\beta+1)}{\Gamma(\alpha(n-1)+\beta+1)}  t^{\alpha (n-1)+\beta}\\=\frac{(\Gamma(\alpha+1))^{n-1}\Gamma(\beta+1)}{\Gamma(\alpha(n-1)+\beta+1)}D C^{n-1}  t^{\alpha (n-1)+\beta}
\end{multline*}
having utilized integration by parts for the second equality and \eqref{eq:estim} for the first inequality. The sum of the last expression and the first term on the right-hand side of \eqref{eq:inter1} is equal to the term of the sum in \eqref{ineq:first} which corresponds to $i=n-1$. The proof of Lemma \ref{lem:estimate} is complete.
\end{proof}

We proceed with a technical estimate of an analytic flavor.
\begin{lemma}\label{lem:import}
Let $\alpha$, $\beta$, $C$ and $D$ be as in \eqref{eq:estim}. For $t\geq 0$ and positive integer $j$ satisfying
\begin{equation}\label{ineq:condition}
2D\Gamma(\beta+1)j(\alpha(j-1)+\beta+1)^{\alpha-\beta}\leq C\Gamma(\alpha+1)t^{\alpha-\beta},
\end{equation}
the following inequality holds
\begin{equation*} 
\sum_{i=0}^{j-1} {j \choose i}\frac{(\Gamma(\alpha+1))^i (\Gamma(\beta+1))^{j-i}}{\Gamma(\alpha i+\beta (j-i)+1)}(C t^\alpha)^i(D t^\beta)^{j-i} \leq 2D C^{j-1} j\frac{(\Gamma(\alpha+1))^{j-1}\Gamma(\beta+1)}{\Gamma(\alpha (j-1)+\beta+1)} t^{\alpha(j-1) + \beta}.
\end{equation*}
\end{lemma}
\begin{proof}
Write
\begin{multline*}
\sum_{i=0}^{j-1} {j \choose i}\frac{(\Gamma(\alpha+1))^i (\Gamma(\beta+1))^{j-i}}{\Gamma(\alpha i+\beta (j-i)+1)}(C t^\alpha)^i( D t^\beta)^{j-i}\\=D C^{j-1} j\frac{(\Gamma(\alpha+1))^{j-1}\Gamma(\beta+1)}{\Gamma(\alpha (j-1)+\beta+1)}t^{\alpha (j-1)+\beta}\\+\sum_{i=0}^{j-2} {j \choose i}\frac{(\Gamma(\alpha+1))^i (\Gamma(\beta+1))^{j-i}}{\Gamma(\alpha i+\beta (j-i)+1)}(C t^\alpha)^i(D t^\beta)^{j-i}.
\end{multline*}
Assuming that $t$ and $j$ satisfy \eqref{ineq:condition} we intend to show that the second term on the right-hand side does not exceed the first term. 
Indeed, using the inequality
\begin{equation}\label{choose}
{j \choose i} \leq j^{j-i}
\end{equation}
and \eqref{eq:gamma} we infer
\begin{multline*}
\frac{\Gamma(\alpha (j-1)+\beta+1)}{D C^{j-1} j(\Gamma(\alpha+1))^{j-1}\Gamma(\beta+1) t^{\alpha(j-1) + \beta}}\sum_{i=0}^{j-2} {j \choose i}\frac{(\Gamma(\alpha+1))^i (\Gamma(\beta+1))^{j-i}}{\Gamma(\alpha i+\beta (j-i)+1)}(C t^\alpha)^i(D t^\beta)^{j-i}\\\leq \frac{C\Gamma(\alpha+1)t^{\alpha-\beta}}{D\Gamma(\beta+1)j}\sum_{i=0}^{j-2}\Big(\frac{D\Gamma(\beta+1)j}{C\Gamma(\alpha+1)t^{\alpha-\beta}}\Big)^{j-i} \frac{\Gamma(\alpha (j-1)+\beta+1)}{\Gamma(\alpha i+\beta (j-i)+1)}\\\leq  \frac{C\Gamma(\alpha+1)t^{\alpha-\beta}}{D\Gamma(\beta+1)j(\alpha(j-1)+\beta+1)^{\alpha-\beta}}\sum_{i=0}^{j-2}\Big(\frac{D\Gamma(\beta+1)j(\alpha(j-1)+\beta+1)^{\alpha-\beta}}
{C\Gamma(\alpha+1)t^{\alpha-\beta}}\Big)^{j-i}\\\leq \frac{D\Gamma(\beta+1)j(\alpha(j-1)+\beta+1)^{\alpha-\beta}}
{C\Gamma(\alpha+1)t^{\alpha-\beta}}\Big(1-\frac{D\Gamma(\beta+1)j(\alpha(j-1)+\beta+1)^{\alpha-\beta}}
{C\Gamma(\alpha+1)t^{\alpha-\beta}}\Big)^{-1}\leq 1.
\end{multline*}
The penultimate inequality is secured by $\sum_{i=0}^{j-2}(\ldots)^{j-i}\leq \sum_{i\geq 2}(\ldots)^i$. The function $x\mapsto x(1-x)^{-1}$ is increasing on $[0,1)$ and equal to $1$ at $x=1/2$. This fact in combination with \eqref{ineq:condition} justifies the last inequality.
\end{proof}

Now we are ready to provide a simplification of formula \eqref{ineq:first} for convolution powers $n$ and arguments $t$ satisfying \eqref{ineq:condition}, with $j=n$.
\begin{cor}\label{cor:ineq}
Suppose \eqref{eq:estim}. Then, for $t\geq 0$ and positive integer $j$ satisfying \eqref{ineq:condition},
\begin{equation}\label{ineq:2jb1}
|V_j(t)-\rho_j t^{\alpha j}|\leq 
2D C^{j-1} j\frac{(\Gamma(\alpha+1))^{j-1}\Gamma(\beta+1)}{\Gamma(\alpha (j-1)+\beta+1)} t^{\alpha(j-1) + \beta}
\end{equation}
and
\begin{equation}\label{ineq:vj}
V_j(t) \leq 5 \rho_j t^{\alpha j}, \quad t \geq 0.
\end{equation}
\end{cor}
\begin{proof}
Inequality \eqref{ineq:2jb1} is an immediate consequence of Lemmas \ref{lem:estimate} and \ref{lem:import}.

A part of \eqref{ineq:2jb1} reads
\begin{equation*}
V_j(t)\leq \rho_j t^{\alpha j} + 
2D C^{j-1} j\frac{(\Gamma(\alpha+1))^{j-1}\Gamma(\beta+1)}{\Gamma(\alpha (j-1)+\beta+1)} t^{\alpha(j-1) + \beta}.
\end{equation*}
To prove \eqref{ineq:vj} we bound from above the ratio of the terms on the right-hand side of the last inequality
\begin{multline*}
\frac{2D C^{j-1} j(\Gamma(\alpha+1))^{j-1}\Gamma(\beta+1) t^{\alpha(j-1) + \beta}}{ \Gamma(\alpha (j-1)+\beta+1)\rho_j t^{\alpha j}} = \frac{2 D j \Gamma(\beta+1) \Gamma(\alpha j +1)}{ C \Gamma(\alpha+1) t^{\alpha-\beta}  \Gamma(\alpha (j-1)+\beta+1)}\\ \leq \frac{2 D j \Gamma(\beta+1) }{ C \Gamma(\alpha+1) t^{\alpha-\beta} } \cdot (\alpha j+1)^{\alpha-\beta} \leq \frac{2(\alpha j+1)^{\alpha-\beta}}{ (\alpha(j-1)+\beta+1)^{\alpha-\beta}}\\ \leq 2\left( 1+ \frac{\alpha-\beta}{\alpha(j-1)+\beta+1} \right)^\alpha \leq 4.
\end{multline*}
Here, the first inequality follows from Lemma \ref{gammainequality}, the second literally repeats \eqref{ineq:condition} and the last is a consequence of $\frac{\alpha-\beta}{\alpha(j-1)+\beta+1}<1$, which in turn is secured by $\alpha\in (0,1)$ and $\beta\geq 0$.
\end{proof}

\begin{cor}\label{cor:sim}
Suppose \eqref{eq:estim}. Then, for all $\gamma>0$,
\begin{equation}\label{eq:uniform}
\lim_{t\to\infty}\sup_{y\geq \gamma t}\Big|\frac{V_j(y)}{\rho_j y^{\alpha j} 
}-1\Big|=0
\end{equation}
whenever $j=j(t)$ satisfies $j(t) = o(t^{\frac{\alpha - \beta}{\alpha-\beta +1}})$ as $t \to \infty$, and particularly
\begin{equation}\label{eq:sim}
V_j(t)~\sim~ \rho_j
t^{\alpha j},  \quad t \to \infty.
 \end{equation}
\end{cor}
\begin{proof}
If $j(t) = o(t^{\frac{\alpha - \beta}{\alpha-\beta +1}})$ as $t \to \infty$, then \eqref{ineq:condition} holds for large $j$ and $t$ and
\begin{multline*}
\sup_{y\geq \gamma t}\frac{D C^{j-1} j\frac{(\Gamma(\alpha+1))^{j-1}\Gamma(\beta+1)}{\Gamma(\alpha (j-1)+\beta+1)} y^{\alpha(j-1) + \beta}}{\rho_j y^{\alpha j}}\leq \frac{D C^{j-1} j\frac{(\Gamma(\alpha+1))^{j-1}\Gamma(\beta+1)}{\Gamma(\alpha (j-1)+\beta+1)}}{\rho_j t^{\alpha-\beta}}\\ ~\sim~ \frac{D\Gamma(\beta+1)\alpha^{\alpha-\beta}j^{\alpha-\beta+1}}{C\Gamma(\alpha+1)t^{\alpha-\beta}}~\to~0
\end{multline*}
as $t\to\infty$. Here, we have used a standard asymptotic relation
\begin{equation}\label{eq:ratio}
\Gamma(x+a)/\Gamma(x)\sim x^a,\quad x\to\infty
\end{equation}
for fixed $a>0$, with $x=1+\alpha j-(\alpha-\beta)$ and $a=\alpha-\beta$. Thus, according to \eqref{ineq:2jb1}, the claim follows.
\end{proof}

\section{Proof of Theorem \ref{main} }\label{sect:proofs}

We start by collecting several (additional) auxiliary results to be used in the proof of Theorem \ref{main}.

Recall that $f:[0,\infty) \to [0,\infty)$ is called {\it directly Riemann integrable} (dRi) function on $[0,\infty)$, if $\overline{\sigma}(h)<\infty$ for each $h>0$ and $\lim_{h \to 0+}(\overline{\sigma}(h) -\underline{\sigma}(h))= 0$, where
$$
\overline{\sigma}(h):= h \sum_{n\geq 1} \sup_{(n-1)h \leq y <nh} f(y) \quad \text{ and } \quad \underline{\sigma}(h):= h \sum_{n\geq 1} \inf_{(n-1)h \leq y <nh} f(y).
$$
A function $f:\mr \to [0,\infty)$ is called dRi on $\mr$, if the same two conditions hold, with $n\in\mathbb{Z}$ replacing $n\geq 1$ in the definition of the integral sums.

The next result follows from the proof of Lemma 4.5 in \cite{Iksanov+Marynych+Samoilenko:2022}.
\begin{lemma}\label{lemma45}
Let $f:[0,\infty) \to [0,\infty)$ be dRi on $[0,\infty)$ and $j$ a positive integer, possibly dependent on $t$ and possibly divergent to $\infty$ together with $t$.
Then
$$ 
\int_{[0,\,t]} f(t-y) { \rm d} V_j(y) = O \left( V_{j-1}(t) \right), \quad t\to \infty.$$
\end{lemma}
\begin{lemma}\label{lemma46}
Suppose \eqref{eq:estim} and let $j=j(t)=o(t^{\frac{\alpha-\beta}{\alpha -\beta +1}})$ as $t \to \infty$. Then
$$
\int_{(t,\, \infty)} \eee^{t-y} { \rm d} V_j(y) = O \left( V_{j-1}(t) \right), \quad t\to \infty.
$$
\end{lemma}
\begin{proof}
Let  $h: \mathbb{R} \to [0,\infty)$ be a dRi function on $\mr$ satisfying $h(t) = 0$ for $t>0$. We start as in the proof of Lemma 4.6
in \cite{Iksanov+Marynych+Samoilenko:2022}: for $t\geq 0$,
\begin{equation*}
\int_{(t,\,\infty)} h(t-y) {\rm d} V_j(y) = 
\int_{[0,\,t]} h_1(t-y) {\rm d} V_{j-1}(y) + \int_{(t,\,\infty)} h_2(t-y) {\rm d} V_{j-1}(y),
\end{equation*}
where $h_1(t) = \int_{(t,\,\infty)}h(t-y){\rm d} V(y)$ and $h_2(t) = \int_{[
0,\,\infty)}h(t-y){\rm d}V(y)$ for $t \in \mathbb{R}$. By Lemma A.1 in \cite{Iksanov+Marynych+Samoilenko:2022}, $h_1(t)\leq b$ for some $b>0$ and all $t\geq 0$, whence
$$
\int_{[0,\,t]} h_1(t-y) {\rm d} V_{j-1}(y) = O \left( V_{j-1}(t) \right), \quad t\to \infty.
$$
Now we put $h(t) = \eee^t \1_{(-\infty,0]} (t)$ and note that all the formulae given in the preceding part of the present proof hold for this $h$.
Plainly, $h_2(t) = \eee^t \int_{[0,\,\infty)} \eee^{-y} {\rm d} V(y)= \kappa \eee^t$ for $t \leq 0$, where $\kappa:= \me \eee^{-\eta} (1- \me \eee^{-\xi})^{-1}$. Integrating by parts yields
\begin{multline*}
\int_{(t,\,\infty)} h_2(t-y) {\rm d} V_{j-1}(y)\\=\kappa \int_{(t,\,\infty)} \eee^{t-y} {\rm d} V_{j-1}(y) = -\kappa V_{j-1}(t) + \kappa\int_t^\infty \eee^{t-y} V_{j-1}(y) {\rm d} y=: \kappa C_j(t)
\end{multline*}
for $t \geq 0$. By Corollary \ref{cor:sim}, \eqref{eq:estim} together with our choice of $j=j(t)$ ensures \eqref{eq:uniform}. In view of \eqref{eq:uniform}, given $\varepsilon >0$ we obtain,
for large enough $t$,
\begin{multline*}
0 \leq C_j(t) \leq -V_{j-1}(t) + (1 + \varepsilon)\rho_{j-1}\int_t^\infty \eee^{t-y} y^{\alpha(j-1)} {\rm d} y =-V_{j-1}(t)\\+(1+\varepsilon)(C \Gamma(\alpha+1))^{j-1}\Big(\sum_{k=0}^{\lfloor \alpha(j-1)\rfloor} \frac{t^{\alpha(j-1)-k}}{\Gamma(\alpha(j-1)+1-k)}+\frac{1}{\Gamma(\{\alpha(j-1)\})}\int_t^\infty \eee^{t-y}y^{\{\alpha(j-1)\}-1}{\rm d}y \Big),
\end{multline*}
where $\{x\}$ denotes the fractional part of $x\in\mr$. The expression in the parantheses is asymptotically equivalent to $t^{\alpha(j-1)}/\Gamma(\alpha(j-1)+1)$ whenever $j(t)=o(t)$ as $t\to\infty$. Hence, recalling that, according to \eqref{eq:sim} (which holds according to Corollary \ref{cor:sim}), $V_{j-1}(t)\sim \rho_{j-1}t^{\alpha(j-1)}$ as $t\to\infty$ we conclude that ${\lim\sup}_{t\to\infty}(C_j(t)/V_{j-1}(t))\leq \varepsilon$ and thereupon $C_j(t)=o(V_{j-1}(t))$ as $t\to\infty$ upon sending $\varepsilon \to 0+$.
\end{proof}
\begin{lemma}\label{lem:aux}
Suppose \eqref{eq:estim} and let $j=j(t) \to \infty$ and $j(t) = o(t^{\frac{\alpha-\beta}{\alpha -\beta +1}})$ as $t \to \infty$. Then
$$\lim_{s \to \infty} \limsup_{t\to \infty}  \frac{j^\alpha}{\rho_{ j - 1} t^{\alpha j}} \int_{(st/j,\,t]} y^\alpha  {\rm d } \big( -V_{j-1}(t-y)\big)= 0.$$
\end{lemma}
\begin{proof}
Integrating by parts we conclude that
\begin{multline*}
\frac{j^\alpha}{\rho_{j-1} t^{\alpha j}} \int_{(st/j,\, t]} y^\alpha  {\rm d } \big( -V_{j-1}(t-y) \big)= \frac{s^\alpha V_{j-1}(t(1-s/j))}{\rho_{j-1} t^{\alpha(j-1)}}\\+
\frac{\alpha j^\alpha}{\rho_{j-1}t^{\alpha j}}\int_{st/j}^t V_{j-1}(t-y) y^{\alpha-1}{\rm d}y.
\end{multline*}
In view of \eqref{eq:uniform} which holds true according to Corollary \ref{cor:sim},
$$
\frac{s^\alpha V_{j-1}(t(1-s/j))}{\rho_{j-1} t^{\alpha (j-1)}}~\sim~s^\alpha  (1-s/j)^{\alpha(j-1)}~\to~ s^\alpha \eee^{-\alpha s},\quad t\to \infty.
$$
The right-hand side converges to $0$ as $s\to\infty$. Further, we infer with the help of \eqref{ineq:first} that
\begin{multline}\label{ineq:lemma5}
\frac{j^\alpha}{\rho_{j-1}t^{\alpha j}}\int_{st/j}^t V_{j-1}(t-y) y^{\alpha-1} {\rm d} y \leq  \frac{j^\alpha}{\rho_{j-1}t^{\alpha j}}\biggl(\int_{st/j}^t \rho_{j-1} (t-y)^{\alpha(j-1)} y^{\alpha-1}{\rm d}y\\
+\int_{st/j}^t\sum_{i=0}^{j-2} {j-1 \choose i}\frac{(\Gamma(\alpha+1))^i (\Gamma(\beta+1))^{j-i-1}}{\Gamma(\alpha i+\beta (j-i-1)+1)}(C(t-y)^\alpha)^i(D(t-y)^\beta)^{j-i-1} y^{\alpha-1} {\rm d}y\biggr).
\end{multline}
Changing the variable we obtain for the first term
\begin{multline*}
\frac{j^\alpha}{t^{\alpha j}} \int_{st/j}^t  (t-y)^{\alpha(j-1)} y^{\alpha-1} {\rm d}y= \int_s^j(1-y/j)^{\alpha(j-1)}y^{\alpha-1}{\rm d}y~ \to~ \int_s^\infty \eee^{-\alpha y}  y^{\alpha-1}{\rm d}y,\quad t\to\infty
\end{multline*}
having utilized the Lebesgue dominated convergence theorem for the limit relation. The right-hand side
converges to $0$ as $s \to \infty$. Invoking $y^{\alpha-1} \leq (st/j)^{\alpha-1}$ for $y\in [st/j, t]$ and then evaluating the remaining integral we conclude that the second term in \eqref{ineq:lemma5} does not exceed
\begin{multline*}
\frac{j^\alpha }{\rho_{ j - 1} t^{\alpha j}}\left(\frac{st}{j} \right)^{\alpha-1}  \sum_{i=0}^{j-2} {j-1 \choose i}   \frac{(C\Gamma(\alpha+1))^i (D\Gamma(\beta+1))^{j-i-1}}{\Gamma(\alpha i+\beta (j-i-1)+1)}
\frac{ ( t- st/j )^{\alpha i + \beta(j-i-1) + 1}}{\alpha i + \beta(j-i-1) + 1}\\= s^{\alpha-1}j \sum_{i=0}^{j-2} {j-1 \choose i} \frac{\Gamma(\alpha(j-1) + 1)}{\Gamma(\alpha i + \beta(j-i-1) + 1)} \left( \frac{D \Gamma(\beta+1) t^\beta}{C \Gamma(\alpha + 1) t^\alpha} \right)^{j-i-1}\frac{( 1- s/j )^{\alpha i + \beta(j-i-1) + 1}}{\alpha i + \beta(j-i-1) + 1}\\\leq \frac{s^{\alpha - 1}}{\beta}  \sum_{i=0}^{j-2} {j-1 \choose i} \frac{\Gamma(\alpha(j-1) + 1)}{\Gamma(\alpha i + \beta(j-i-1) + 1)} \left( \frac{D \Gamma(\beta+1) t^\beta}{C \Gamma(\alpha + 1) t^\alpha} \right)^{j-i-1}\\\leq \frac{ s^{\alpha - 1}}{ \beta}\sum_{i=0}^{j-2} \left((j-1) (\alpha(j-1) + 1)^{\alpha -\beta} \frac{ D \Gamma(\beta+1)}{C \Gamma(\alpha + 1) t^{\alpha-\beta}} \right)^{j-i-1}\\ \leq \frac{ s^{\alpha - 1} }{ \beta} \frac{ D \Gamma(\beta+1)}{C \Gamma(\alpha + 1)}(j-1) \Big(\frac{\alpha(j-1) + 1}{t}\Big)^{\alpha -\beta} \left(1-\frac{ D \Gamma(\beta+1)}{C \Gamma(\alpha + 1)}(j-1) \Big(\frac{\alpha(j-1) + 1}{t}\Big)^{\alpha -\beta}\right)^{-1}\\~\to~0,\quad t\to\infty.
\end{multline*}
We have used $0<(1-s/j)^{\alpha i+\beta(j-i-1)+1}\leq 1$ and $\frac{1}{\alpha i+\beta(j-i-1)+1}\leq \frac{1}{\beta j}$ for the first inequality, formula \eqref{choose} and Lemma \ref{gammainequality} for the second. Our assumption $j(t) = o(t^{\frac{\alpha-\beta}{\alpha -\beta +1}})$ as $t \to \infty$ entails $\lim_{t\to\infty} (j(t)-1) ((\alpha(j(t)-1) + 1)/t)^{\alpha -\beta}=0$, thereby justifying the limit relation.
\end{proof}

A major part of the proof of Theorem \ref{main} is covered by Theorems \ref{theorem2} and \ref{theorem3} given next.
\begin{theorem}\label{theorem2}
Suppose \eqref{eq:estim} and let $j=j(t) \to \infty $ and $j(t) = o(t^{\frac{\alpha-\beta}{\alpha -\beta +1}})$ as $t \to \infty$. Then
$$
\left(\frac{c (j(t))^\alpha}{\rho_{\lfloor j(t) u \rfloor - 1} t^{\alpha \lfloor j(t) u \rfloor}}\sum_{r \geq 1} V_{\lfloor j(t) u \rfloor - 1}(t - T_r) \1_{\{T_r\leq t\}}\right)_{u>0}~\overset{{\rm f.d.d.}}{\longrightarrow}~ \left(\int_{[0,\,\infty)}\eee^{-\alpha u y} {\rm d} \mathcal{Z}_\alpha^\leftarrow(y) \right)_{u>0}, \quad t \to \infty,
$$
where $\rho_n$ is as given in \eqref{def:fi}, and $\mathcal{Z}_\alpha^\leftarrow$ is an inverse $\alpha$-stable subordinator.
\end{theorem}

Our proof of Theorem \ref{theorem2} uses an auxiliary technical result, which is a slight reformulation of Lemma A.5 in \cite{IKSANOV2013}. Let $ \mathcal{D}$ denote the Skorokhod space of c\`{a}dl\`{a}g functions defined on $[0,\infty)$. We assume that the space $\mathcal{D}$ is endowed with the $J_1$-topology. Comprehensive information about the $J_1$-topology on $\mathcal{D}$ can be found in \cite{billingsley}.
\begin{lemma}\label{lemma6}
For each $k \in \mn$, let $y_k : [0,\infty) \to \mr
$ be a right-continuous bounded and nondecreasing function. Assume that $\lim_{k\to \infty} x_k = x$
on $\mathcal{D}$ and that, for each $t\geq 0$, $\lim_{k\to \infty}y_k(t) = y(t)$, where $y : [0,\infty) \to \mr
$ is a bounded continuous function. Then, for all $a,b\geq 0$, a<b,
$$\lim_{k\to \infty} \int_{[a,\,b]} x_k(t) {\rm d} y_k(t) = \int_{[a,\,b]} x(t) {\rm d} y(t).$$
\end{lemma}
\begin{proof}[Proof of Theorem \ref{theorem2}]
For notational simplicity we write $j$ for $j(t)$. In view of the Cram\'{e}r-Wold device and the equality $\int_{[0,\,\infty)}\eee^{-\alpha u y} {\rm d} \mathcal{Z}_\alpha^\leftarrow(y)=\alpha u \int_0^\infty \mathcal{Z}_\alpha^\leftarrow(y) \eee^{-\alpha u y}{\rm d}y$ for $u>0$, obtained with the help of integration by parts, it suffices to show that, for any $\ell \in \mn$, any
positive $\lambda_1,\ldots, \lambda_\ell$ and any $0<u_1< \ldots <u_\ell$,
$$
\sum_{i=1}^\ell \lambda_i \frac{c  j^\alpha \sum_{r \geq 1} V_{\lfloor j u_i \rfloor - 1}(t - T_r) \1_{\{T_r\leq t\}}}{ \rho_{\lfloor j u_i \rfloor - 1} t^{\alpha \lfloor j u_i \rfloor}} \overset{{\rm d}}{\longrightarrow}  \sum_{i=1}^\ell \lambda_i \alpha u_i \int_0^\infty \mathcal{Z}_\alpha^\leftarrow(y) \eee^{-\alpha u_i y} {\rm d} y,\quad t\to\infty,
$$
where, as before, $\overset{{\rm d}}{\longrightarrow}$ denotes convergence in distribution. Write, for any $s >0$ and sufficiently large $t$,
\begin{multline*}
\frac{cj^\alpha \sum_{r\geq 1} V_{\lfloor j u_i \rfloor - 1}(t-T_r)\1_{\{T_r\leq t\}}}{\rho_{\lfloor j u_i\rfloor-1}t^{\alpha\lfloor ju_i\rfloor}}=\frac{cj^\alpha\int_{[0,\,t]}V_{\lfloor ju_i \rfloor-1}(t-y){\rm d}N(y)}{\rho_{\lfloor j u_i \rfloor - 1} t^{\alpha \lfloor j u_i \rfloor }}\\=\frac{c j^\alpha \int_{[0,\,t]} N(y) {\rm d}_y(- V_{\lfloor j u_i \rfloor - 1}(t - y)) }{ \rho_{\lfloor j u_i \rfloor - 1} t^{\alpha \lfloor j u_i \rfloor }}\\ = \frac{1}{\rho_{ \lfloor j u_i \rfloor - 1} t^{\alpha( \lfloor j u_i \rfloor - 1)}} \int_{[0,\,s]}\frac{c N(yt/j)}{(t/j)^\alpha} {\rm d}_y (-V_{\lfloor j u_i \rfloor - 1}(t(1-y/j)))\\+ \frac{cj^\alpha}{\rho_{ \lfloor j u_i \rfloor - 1} t^{\alpha \lfloor j u_i \rfloor}} \int_{(st/j,\,t]} N(y) {\rm d}_y(-V_{\lfloor j u_i \rfloor-1}(t-y)).
\end{multline*}
Using \eqref{eq:uniform} we obtain, for each fixed $y>0$, as $t\to\infty$,
\begin{equation}\label{sim_exp}
V_{\lfloor j u_i \rfloor - 1} \left(t\left(1-\frac{y}{j}\right)\right)~\sim~ \rho_{\lfloor j u_i \rfloor - 1} t^{\alpha( \lfloor j u_i \rfloor - 1 )} \left(1 - \frac{y}{j} \right)^{\alpha(\lfloor j u_i \rfloor - 1)}~\sim~
\rho_{\lfloor j u_i \rfloor - 1} t^{\alpha( \lfloor j u_i \rfloor - 1 )} \eee^{-\alpha u_i y}.
\end{equation}
Under the sole assumption $\mmp\{\xi>t\}\sim ct^{-\alpha}$ as $t\to\infty$, an application of part (B.4) of Theorem 3.2 in \cite{Alsmeyer+Iksanov+Marynych:2017} yields
$$
\left(c(t/j)^{-\alpha} N(yt/j)\right)_{y\geq 0}~\Longrightarrow~\left( \mathcal{Z}_\alpha^\leftarrow(y) \right)_{y\geq 0}, \quad t\to\infty$$
on $\mathcal{D}$, where $\Longrightarrow$ denotes weak convergence on a function space. Here, we have used the fact that the assumptions imposed on $j$ ensure that $\lim_{t\to\infty}(t/j(t))=\infty$.
Let $(t_k)_{k\geq 1}$ be any sequence of positive numbers satisfying $\lim_{k\to\infty} t_k=\infty$. According to the Skorokhod representation theorem, there exist $\mathcal{\hat Z}_\alpha^\leftarrow$ a version of $\mathcal{Z}_\alpha^\leftarrow$ and $((\hat{N}_{t_k}(y))_{y\geq 0})_{k\geq 1}$ a version of
$((c(t_k/j(t_k))^{-\alpha}N(yt_k/j(t_k)))_{y\geq 0})_{k\geq 1}$ such that, for all $T>0$,
$$
\lim_{k\to \infty}\sup_{y \in [0,\,T]} | \hat{N}_{t_k}(y)-\mathcal{\hat Z}_\alpha^\leftarrow(y)| = 0 \quad \text{a.s.}
$$
In view of \eqref{sim_exp},
$$
\lim_{k\to \infty} \frac{V_{\lfloor j(t_k) u_i \rfloor - 1}(t_k(1-y/j(t_k)))}{\rho_{ \lfloor j(t_k) u_i \rfloor - 1} t_k^{\alpha( \lfloor j(t_k) u_i \rfloor - 1)}} = \eee^{-\alpha u_i y}.
$$
Using Lemma \ref{lemma6}, with $x_k= \hat{N}_{t_k}$, $x=\mathcal{\hat Z}_\alpha^\leftarrow$, and $y_k$ and $y$ given by the functions on the left-hand and the right-hand side of the last centered formula, respectively, we obtain
\begin{multline*}
\lim_{k \to \infty} \frac{1}{\rho_{ \lfloor j(t_k) u_i \rfloor - 1} t_k^{\alpha( \lfloor j(t_k) u_i \rfloor - 1)}} \int_{[0,\,s]} \hat{N}_{t_k}(y) {\rm d}_y (-V_{\lfloor j(t_k) u_i \rfloor - 1}(t_k(1-y/j(t_k))))\\= \alpha u_i \int_0^s \mathcal{\hat Z}_\alpha^\leftarrow (y) \eee^{-\alpha u_i y} {\rm d} y \quad \text{a.s.}
\end{multline*}
Since the diverging sequence $(t_k)_{k\geq 1}$ is arbitrary, this entails
\begin{multline*}
\sum_{i=1}^\ell \lambda_i \frac{1}{\rho_{ \lfloor j u_i \rfloor - 1} t^{\alpha( \lfloor j u_i \rfloor - 1)}} \int_{[0,\,s]} c(t/j)^{-\alpha}N(yt/j){\rm d}_y (-V_{\lfloor j u_i \rfloor - 1}(t(1-y/j)))\\~\overset{{\rm d}}{\longrightarrow}~\sum_{i=i}^\ell \lambda_i \alpha u_i \int_0^s \mathcal{Z}_\alpha^\leftarrow (y) \eee^{-\alpha u_i y} {\rm d} y,  \quad t\to\infty.
\end{multline*}
In view of
$$\lim_{s\to\infty}\sum_{i=1}^\ell \lambda_i  \alpha u_i \int_0^s \mathcal{Z}_\alpha^\leftarrow (y) \eee^{-\alpha u_i y} {\rm d}y = \sum_{i=1}^\ell \lambda_i  \alpha u_i \int_0^\infty \mathcal{Z}_\alpha^\leftarrow(y) \eee^{-\alpha u_i y} {\rm d} y  \quad \text{a.s.},$$
we are left with showing that, for all $\varepsilon>0$,
$$
\lim_{s\to \infty} \limsup_{t\to \infty} \mmp \biggl\{\frac{j^\alpha}{\rho_{ \lfloor j u_i \rfloor - 1} t^{\alpha \lfloor j u_i \rfloor}} \int_{(st/j,\, t]} N(y) {\rm d}_y(-V_{\lfloor j u_i \rfloor-1}(t-y))>\varepsilon \biggl\} = 0.
$$
Noting that $\me N(y)=V(y)\leq 2Cy^\alpha$ for large $y$, this limit relation is ensured by Markov`s inequality and Lemma \ref{lem:aux}.
\end{proof}

\begin{theorem}\label{theorem3}
Suppose \eqref{eq:estim} and let $t\mapsto j(t)$ be an integer-valued function satisfying $j(t)=o(t^{1/3})$ as $t\to\infty$ and $\lim_{t\to\infty}j(t)=\infty$. Then
$$\frac{(j(t))^\alpha}{\rho_{j(t) - 1} t^{\alpha j(t)}}\Big(N_{j(t)}(t)-\sum_{r\geq 1} V_{j(t) - 1}(t - T_r) \1_{\{T_r\leq t\}}\Big)~ \overset{\mmp}{\longrightarrow}~ 0, \quad t\to\infty.$$
\end{theorem}
\begin{proof}
Although the scheme of the proof is analogous to that of Theorem 3.1 in \cite{Iksanov+Marynych+Samoilenko:2022}, technical details are essentially different at places.

For $j \in \mn$ and $t\geq 0$, put $D_j(t) := {\rm Var}\, N_j(t)$ and
$$
I_j(t):= \me \left( \sum_{r\geq1} V_{j-1}(t - T_r) \1_{\{T_r\leq t\}} - V_j(t) \right)^2
$$
with the convention that $V_0(t) = 1$ for $t\geq 0$, so that $I_1(t)=D_1(t)$. We shall use a formula: for $j \geq 2$ and $t\geq 0$,
\begin{multline}\label{eq:recursive}
D_j(t) = \me \left( \sum_{r\geq1}(N_{j-1}^{(r)} (t - T_r) - V_{j-1}(t-T_r)) \1_{{\{T_r\leq t\}}}\right)^2 \\+
\me \left( \sum_{r\geq1} V_{j-1}(t-T_r) \1_{{\{T_r\leq t\}}}  - V_j(t) \right)^2 = \int_{[0,\,t]}D_{j-1} (t-y) {\rm d} V(y) + I_j(t).
\end{multline}
Iterating \eqref{eq:recursive} yields
\begin{equation}\label{eq:intdj}
\int_{[0,\,t]}D_{j-1} (t-y) {\rm d} V(y) = \sum_{k=1}^{j-1} \int_{[0,\,t]}I_k(t-y) {\rm d} V_{j-k} (y), \quad j\geq 2,~ t\geq 0.
\end{equation}
First we show that $I_j$ is bounded from above by a nonnegative and nondecreasing function. We need an inequality obtained in the proof of Theorem 3.1 in \cite{Iksanov+Marynych+Samoilenko:2022}
\begin{multline}\label{eq:ij}
I_j(t)\leq V_{j-1}(t)V_j(t) + 2 \int_{[0,\,t]} V_{j-1}(t-y) V_j(t-y){\rm d}  U(y) - (V_j(t))^2\\\leq V_{j-1}(t)V_j(t) + 2 \int_{[0,\,t]} V_{j-1}(t-y) V_j(t-y){\rm d}U(y).
\end{multline}
Put $\tilde{U} (x) := \sum_{i\geq 1} \mmp\{ S_i\leq x \}$ for $x\in\mr$ and note that $U(x) = \tilde{U} (x)+ 1$ for $x \geq 0$. Assumption $A$ together with $\tilde U(0)=0$ entails
\begin{equation*}
\tilde{U}(x)\leq Cx^\alpha+C_1x^\rho,\quad x\geq 0
\end{equation*}
for some constant $C_1>0$. Integrating by parts and then using the last inequality yields
\begin{multline*}
\int_{[0,\,t]} V_{j-1}(t-y)V_j(t-y) {\rm d} U(y)= V_{j-1}(t)V_j(t)+ \int_{[0,\,t]} V_{j-1}(t-y)V_j(t-y) {\rm d} \tilde{U}(y)\\ = V_{j-1}(t)V_j(t)+\int_{[0,\,t]} \tilde{U}(t-y) {\rm d}(V_{j-1}(y) V_j(y))\\
\leq V_{j-1}(t)V_j(t)+ C\int_{[0,\,t]}(t-y)^\alpha {\rm d}(V_{j-1}(y) V_j(y))+C_1\int_{[0,\,t]}(t-y)^\rho {\rm d}(V_{j-1}(y) V_j(y)).
\end{multline*}
With this at hand, integrating by parts once again we arrive at
\begin{multline}\label{def:hj}
I_j(t) \leq V_{j-1}(t) V_j(t) + 2\int_{[0,\,t]}V_{j-1}(t-y)V_j(t-y) {\rm d} U(y) \leq 3 V_{j-1}(t) V_j(t)\\+ 2 C \alpha \int_0^t V_{j-1}(y)V_j(y)(t-y)^{\alpha-1} {\rm d}y+2 C_1\rho \int_0^t V_{j-1}(y)V_j(y)(t-y)^{\rho-1} {\rm d} y=:h_j(t).
\end{multline}
The function $h_j$ is a nonnegative and nondecreasing function that we were looking for. Combining \eqref{def:hj} with \eqref{eq:recursive} and \eqref{eq:intdj} we infer
\begin{multline}\label{ineq:djhj}
D_{j-1}(t) =
\sum_{k=1}^{j-2 } \int_0^t I_k(t-y) {\rm d} V_{j-k-1}(y) + I_{j-1}(t)\\\leq \sum_{k=1}^{j-2} h_k(t) V_{j-k-1}(t)+h_{j-1}(t), \quad j \geq 2,~t\geq 0.
\end{multline}

Next, we obtain an upper bound for $h_j$ (and for $D_j$) which is valid for large arguments. Fix $j\in\mn$  and $s\geq0$ satisfying
\begin{equation}\label{scondition}
2D \Gamma(\beta+1) j (\alpha(j-1)+\beta+1)^{\alpha-\beta} \leq C\Gamma(\alpha+1)s^{\alpha-\beta}
\end{equation}
and also
\begin{equation}\label{jcondition}
2(\alpha(j-2) +1)^{2\alpha} \leq C \Gamma(\alpha+1) s^\alpha,
\end{equation}
\begin{equation}\label{cond_1}
3(2\alpha j + 1)^\alpha \leq 2 C \Gamma(\alpha+1)s^\alpha,
\end{equation}
\begin{equation}\label{cond_2}
C_1 \Gamma(\rho)(2\alpha j + 1)^{\alpha-\rho} \leq  \Gamma(\alpha+1) s^{\alpha-\rho},
\end{equation}
and
\begin{equation}\label{cond_3}
j\geq 1/\alpha+2.
\end{equation}
Plainly, these inequalities trivially hold for large enough $s$ whenever $j=j(s)=o(s^{1/2})$ as $s\to\infty$. 

Put $r:=5^2$. In view of \eqref{scondition}, we conclude with the help of formula \eqref{ineq:vj} that, for  $1\leq k \leq j$,
$$
V_{k-1}(s)V_k(s)\leq 
r \rho_{k-1} \rho_k s^{\alpha(2k-1)}.
$$
Substituting this into \eqref{def:hj} we obtain
\begin{multline}\label{gre:hk}
h_k(s) \leq  3r \rho_{k-1} \rho_k s^{\alpha(2k-1)}\\+ 2C\alpha r \rho_{k-1} \rho_k \int_0^s  s^{\alpha(2k-1)}(s-y)^{\alpha-1}{\rm d}y+2C_1\rho r \rho_{k-1} \rho_k\int_0^s  s^{\alpha(2k-1)}(s-y)^{\rho-1} {\rm d} y\\=3r \rho_{k-1} \rho_k s^{\alpha(2k-1)} + 2r C \alpha \rho_{k-1}\rho_k {\rm B}(\alpha(2k-1) + 1, \alpha)s^{2\alpha k}\\+ 2r C_1 \rho \rho_{k-1} \rho_k {\rm B}(\alpha(2k-1)+1, \rho) {\rm B}(\alpha(2k-1)+1,\rho)s^{2\alpha k-\alpha+\rho},
\end{multline}
where ${\rm B}$ is the Euler beta-function. The second term on the right-hand side is of the highest order. Сonsider the ratio of the first and the second terms in \eqref{gre:hk}:
\begin{multline*}
\frac{ 3 r \rho_{k-1} \rho_k s^{\alpha(2k-1)}}{2r C \alpha \rho_{k-1}\rho_k {\rm B} (\alpha(2k-1) + 1, \alpha)s^{2\alpha k}} = \frac{ 3 \Gamma(2\alpha k +1)}{2C \Gamma(\alpha+1)\Gamma(\alpha(2k-1)+1)s^\alpha}\\
\leq \frac{3(2\alpha k+1)^\alpha}{2 C \Gamma(\alpha+1) s^\alpha} \leq \frac{3(2\alpha j+1)^\alpha}{2 C \Gamma(\alpha+1) s^\alpha} \leq 1,
\end{multline*}
where the first inequality is justified by Lemma \ref{gammainequality} and the last inequality is secured by \eqref{cond_1}. The ratio of the third and the second terms in \eqref{gre:hk} can be estimated as follows:
\begin{multline*}
\frac{2r C_1 \rho \rho_{k-1} \rho_k {\rm B}(\alpha(2k-1)+1, \rho)s^{2\alpha k - \alpha + \rho} }{2rC \alpha \rho_{k-1}\rho_k{\rm B}(\alpha(2k-1) + 1, \alpha)s^{2\alpha k}} = \frac{C_1 \Gamma(\rho) \Gamma(2\alpha k+1)}{\Gamma(\alpha+1)s^{\alpha-\rho} \Gamma(2\alpha k -\alpha+\rho+1)}\\\leq \frac{ C_1 \Gamma(\rho) (2\alpha k +1)^{\alpha-\rho}}{ \Gamma(\alpha+1) s^{\alpha-\rho}} \leq \frac{ C_1 \Gamma(\rho) (2\alpha j +1)^{\alpha-\rho}}{ \Gamma(\alpha+1) s^{\alpha-\rho}} \leq 1,
\end{multline*}
where the first inequality follows from Lemma \ref{gammainequality} and the last inequality is ensured by \eqref{cond_2}. Inequality \eqref{gre:hk} in combination with the two estimates enables us to conclude that, for $1\leq k\leq j$,
\begin{equation*}
h_k(s) \leq 6r C \alpha \rho_{k-1}\rho_k {\rm B} (\alpha(2k-1) + 1, \alpha) s^{2\alpha k}  =: C_2\rho_{k-1}\rho_k {\rm B}(\alpha(2k-1) + 1, \alpha)s^{2\alpha k}.
\end{equation*}
Invoking \eqref{ineq:djhj} we further obtain
\begin{multline}\label{ineq:Dj}
D_{j-1}(s) \leq C_2 \rho_{j-2}\rho_{j-1}{\rm B}(\alpha(2j-3)+1, \alpha)s^{2\alpha (j-1)}\\+\sum_{k=1}^{j-2} C_2 \rho_{k-1}\rho_k{\rm B}(\alpha(2k-1)+1, \alpha)s^{2\alpha k} V_{j-k-1}(s).
\end{multline}
Write with the help of \eqref{def:fi} and \eqref{ineq:vj}
\begin{multline}\label{sum:ineq}
\sum_{k=1}^{j-2}  \rho_{k-1}\rho_k {\rm B}(\alpha(2k-1)+1, \alpha)s^{2\alpha k} V_{j-k-1}(s)\\ \leq 
5 \sum_{k=1}^{j-2}  \rho_{k-1}\rho_k {\rm B} (\alpha(2k-1)+1, \alpha)s^{2\alpha k} \rho_{j-k-1} s^{\alpha(j-k-1)}\\= 5\Gamma(\alpha) \sum_{k=1}^{j-2}  \frac{ (C\Gamma(\alpha +1))^{j+k-2} \Gamma(\alpha(2k-1) +1) }{\Gamma(\alpha(k-1)+1) \Gamma(\alpha k+1) \Gamma(\alpha(j-k-1)+1) \Gamma(2\alpha k+1)}s^{\alpha(j+k-1)}\\\leq 5(2\alpha + 1)\Gamma(\alpha)\sum_{k=1}^{j-2} \frac{(C \Gamma(\alpha +1))^{j+k-2}}{(\Gamma(\alpha(k-1)+1))^2\Gamma(\alpha(j-k-1)+1)}s^{\alpha(j+k-1)}.
\end{multline}
To justify the last inequality, we use the fact that the gamma-function $\Gamma$ is increasing on $[2,\infty)$, whence, for $k\geq 1$,
\begin{multline*}
\frac{1}{\Gamma(\alpha k + 1)} = \frac{\alpha k + 1}{\Gamma(\alpha k + 2)} \leq \frac{\alpha k + 1}{\Gamma(\alpha(k-1) + 2)} = \frac{\alpha k + 1}{(\alpha(k-1) + 1)\Gamma(\alpha (k-1) + 1)}\\\leq  \frac{\alpha + 1}{\Gamma(\alpha (k-1) + 1)}
\end{multline*}
and
$$
\frac{\Gamma(\alpha(2k-1) + 1)}{\Gamma(2\alpha k + 1)} =
\frac{\Gamma(\alpha(2k-1) + 2)}{\Gamma(2\alpha k + 2)} \frac{2\alpha k + 1}{\alpha(2k-1) + 1} \leq \frac{2\alpha k + 1}{\alpha(2k-1) + 1}\leq \frac{2\alpha  + 1}{\alpha + 1}.
$$
The right-hand side of \eqref{sum:ineq}, with the multiplicative constant omitted, can be bounded from above as follows:
\begin{multline}\label{sum:ineq2}
\frac{1 
}{(\Gamma(\alpha(j-2)+1))^2} \sum_{k=1}^{j-2}\frac{(C \Gamma(\alpha +1))^{j+k-2} (\Gamma(\alpha(j-2)+1))^2}{(\Gamma(\alpha(k-1)+1))^2  \Gamma(\alpha(j-k-1)+1)}s^{\alpha(j+k-1)}\\
\leq \frac{{\rm const} 
}{(\Gamma(\alpha(j-2)+1))^2} \sum_{k=1}^{j-2}(C \Gamma(\alpha +1))^{j+k-2} (\alpha(j-2) + 1)^{2\alpha(j-k-1)} s^{\alpha(j+k-1)}.
\end{multline}
Here and hereafter, ${\rm const}$ denotes a constant whose value is of no importance and may change from one appearance to another. The last inequality is a consequence of
$$
\frac{\Gamma(\alpha(j-2)+1)}{\Gamma(\alpha(k-1)+1)} \leq 
(\alpha(j-2)+1)^{\alpha(j-k-1)},
$$
which follows from Lemma \ref{gammainequality}, and
$
\frac{1}{\Gamma(\alpha(j-k-1)+1)} \leq \frac{1}{\min_{z\in [1,2]}\,\Gamma(z)}<\infty
$
for  $1\leq k \leq j-2$ (the minimum cannot be equal to $0$ because $\Gamma(z)$ is the moment of order $z-1$ of an exponentially distributed random variable with mean one). We proceed by bounding from above the right-hand side of \eqref{sum:ineq2}, up to the multiplicative constant:
\begin{multline*} 
 \frac{
(C \Gamma(\alpha +1))^{2j-3} s^{2\alpha(j-1)}}{(\Gamma(\alpha(j-2)+1))^2}\sum_{k=1}^{j-2} \left( \frac{(\alpha(j-2)+1)^{2\alpha}}{C \Gamma(\alpha+1) s^\alpha} \right)^{j-k-1}\\=\frac{(C \Gamma(\alpha +1))^{2j-3} s^{2\alpha(j-1)}}{(\Gamma(\alpha(j-2)+1))^2}\sum_{k=1}^{j-2} \left(\frac{(\alpha(j-2)+1)^{2\alpha}}{C \Gamma(\alpha+1) s^\alpha} \right)^k \\ \leq \frac{(C \Gamma(\alpha +1))^{2j-3} s^{2\alpha(j-1)}}{(\Gamma(\alpha(j-2)+1))^2} \sum_{k\geq 1}\left( \frac{(\alpha(j-2)+1)^{2\alpha}}{C \Gamma(\alpha+1) s^\alpha} \right)^k\leq \frac{(C \Gamma(\alpha +1))^{2j-3} s^{2\alpha(j-1)}}{(\Gamma(\alpha(j-2)+1))^2}.
\end{multline*}
Here, the last inequality is secured by \eqref{jcondition}.
Combining this with \eqref{ineq:Dj} we conclude that
\begin{equation*}
D_{j-1}(s) \leq C_2 \rho_{j-2}\rho_{j-1}{\rm B}(\alpha(2j-3)+1, \alpha)s^{2\alpha (j-1)}+ \frac{{\rm const} 
(C \Gamma(\alpha +1))^{2j-3}s^{2\alpha(j-1)}}{(\Gamma(\alpha(j-2)+1))^2}.
\end{equation*}
Condition \eqref{cond_3} entails $2j\geq 1/\alpha+3$. This together with monotonicity of the gamma-function on $[2,\infty)$ proves 
$$
\frac{\Gamma(\alpha)(C \Gamma(\alpha +1))^{2j-3}}{\Gamma^2(\alpha(j-2)+1) \rho_{j-2} \rho_{j-1} {\rm B}(\alpha(2j-3) + 1, \alpha)} = \frac{\Gamma(\alpha(j-1)+1) \Gamma(\alpha(2j-2)+1)}{\Gamma(\alpha(j-2)+1) \Gamma(\alpha(2j-3)+1)} \geq 1
$$
and thereupon 
\begin{equation}\label{ineq:Dj_2}
D_{j-1}(s) \leq {\rm const}\, 
\frac{ 
(C \Gamma(\alpha +1))^{2j-3}s^{2\alpha(j-1)}}{(\Gamma(\alpha(j-2)+1))^2}.
\end{equation}
Using \eqref{ineq:Dj_2} we infer
\begin{multline}\label{int:estimation}
\int_{[0,\,t]} D_{j-1}(t-y) {\rm d} V(y) = \int_{[0,\, t-j]} D_{j-1}(t-y){\rm d} V(y) + \int_{(t-j,\, t]} D_{j-1}(t-y){\rm d} V(y)\\ \leq {\rm const}\,\frac{(C \Gamma(\alpha+1))^{2j-3}}{(\Gamma(\alpha(j-2)+1))^2}\int_{[0,\,t]} (t-y)^{2\alpha (j-1)} {\rm d} V(y) + \max_{s \in [0,\, j]} D_{j-1}(s)\, U(j)
\end{multline}
having utilized $V(t)-V(t-j)\leq U(j)$, see formula (40) in \cite{Bohunetal:2022} and its proof. Integrating by parts and then invoking \eqref{eq:estim} we estimate the last integral as follows:
\begin{multline*}
\int_{[0,\,t]} (t-y)^{2\alpha (j-1)} {\rm d} V(y)\leq  2\alpha (j-1)  \int_0^t (Cy^\alpha + Dy^\beta)(t-y)^{2\alpha (j-1) -1} {\rm d}y\\=2\alpha (j-1)(C{\rm B}(2\alpha(j-1), \alpha+1)t^{\alpha(2j-1)}+ D {\rm B}(2\alpha(j-1),\beta+1)t^{2\alpha(j-1) + \beta}\\
= O \left(j {\rm B}(2\alpha(j-1),\alpha+1)t^{\alpha(2j-1)}\right),\quad t\to\infty.
\end{multline*}
To justify the last equality, note that, as $j,t\to\infty$, $${\rm B}(2\alpha(j-1), \alpha+1)t^{\alpha(2j-1)}~\sim~ {\rm const}\,\frac{t^{\alpha(j-1)}}{j^{\alpha+1}},\quad {\rm B}(2\alpha(j-1), \beta+1)t^{2\alpha(j-1)+\beta}~\sim~{\rm const}\,\frac{t^{2\alpha(j-1)+\beta}}{j^{\beta+1}}.$$ Hence, ${\rm B}(2\alpha(j-1), \beta+1)t^{2\alpha(j-1)+\beta}=o({\rm B}(2\alpha(j-1), \alpha+1)t^{\alpha(2j-1)})$ as a consequence of $j=j(t)=o(t^{1/3})$. Thus, the first term on the right-hand side of \eqref{int:estimation} is of the order
$$
\frac{(C \Gamma(\alpha+1 ))^{2j-3}}{(\Gamma(\alpha(j-2)+1))^2}j {\rm B}(2\alpha(j-1),\alpha+1)t^{\alpha(2j-1)}.
$$
The product of this and the squared normalization (appearing in the theorem) vanishes under the assumption $j=j(t)=o(t^{1/3})$ as $t\to\infty$:
\begin{multline*}
\frac{(C \Gamma(\alpha+1))^{2j-3}j {\rm B}(2\alpha(j-1),\alpha+1)t^{\alpha(2j-1)}}{(\Gamma(\alpha(j-2)+1))^2}\frac{j^{2\alpha}}{\rho_{j-1}^2 t^{2\alpha j}}\\
= \frac{\Gamma(2\alpha(j-1))}{\Gamma(\alpha(2j-1)+1)} \frac{(\Gamma(\alpha(j-1)+1))^2}{(\Gamma(\alpha (j-2)+1))^2} \frac{j^{2\alpha+1}}{C t^\alpha}~\sim~ \frac{\alpha^{2\alpha}}{C2^{\alpha+1}}\frac{j^{3\alpha}}{t^\alpha}~\to~ 0,\quad t\to\infty.
\end{multline*}
Further, using \eqref{ineq:Dj_2} and \eqref{ineq:vj},
\begin{multline}\label{ineq:max}
\max_{s \in [0,\, j]}\, D_{j-1}(s) U(j) \leq \left(D_{j-1}(j) + (V_{j-1}(j))^2\right) U(j)\\ \leq \left({\rm const}\,\frac{(C \Gamma(\alpha +1))^{2j-3}j^{2\alpha(j-1)}}{(\Gamma(\alpha(j-2)+1))^2}+r \rho_{j-1}^2 j^{2 \alpha (j-1)} \right) U(j).
\end{multline}
Assumption $A$ entails $U(j)\sim Cj^\alpha$ as $j\to\infty$. Now we show that the right-hand side of \eqref{ineq:max} times the squared normalization vanishes whenever $j=j(t)=o(t)$ and $j(t)\to\infty$ as $t\to\infty$. We start with the first term: 
\begin{multline*}
\frac{j^{2\alpha j} (C \Gamma(\alpha +1))^{2j-3}}{\rho_{j-1}^2 t^{2\alpha j} (\Gamma(\alpha(j-2)+1))^2}U(j)=\frac{1}{C\Gamma(\alpha+1)}\left( \frac{j}{t} \right)^{2\alpha j} \frac{(\Gamma(\alpha(j-1)+1))^2}{(\Gamma(\alpha(j-2)+1))^2} U(j)\\ \sim  \frac{\alpha^{2\alpha}}{C\Gamma(\alpha+1)}\left( \frac{j}{t} \right)^{2\alpha j}j^{3\alpha}~\to~ 0, \quad t\to\infty,
\end{multline*}
because $(j/t)^{2\alpha j}$  converges to zero faster than any negative power of $j$. As for the second term in \eqref{ineq:max}, arguing similarly we infer
$$
\frac{j^{2\alpha j}}{t^{2\alpha j}} U(a_j)~\sim~   \left( \frac{j}{t} \right)^{2\alpha j}C j^\alpha~\to~ 0,\quad t\to\infty.
$$
An application of Markov's inequality completes the proof of Theorem \ref{theorem3}.
\end{proof}

We are ready to prove Theorem \ref{main}.
\begin{proof}[Proof of Theorem \ref{main}]
Equality \eqref{eq:stick} is equivalent to
$$
|\log P_k|= |\log W_1| +\ldots + |\log W_{k-1}| + |\log (1-W_k)|,\quad k\in\mn.
$$
Thus,
$ (|\log P_k|)_{k \geq 1}$ is a globally perturbed random walk generated by $(|\log W|, |\log(1-W)|)$, that is, in the notation of Section \ref{sect:assump}, $(\xi, \eta)=(|\log W|, |\log(1-W)|)$. 

For $j\in\mn$ and $t>0$, put $\psi_j(t):=\sum_{|v|=j}\1_{\{P(v)\geq 1/t\}}$ and observe that $\psi_j(n)=N_j(\log n)$ for $j,n\in \mn$. We shall use a decomposition
\begin{multline*}
K_n( \lfloor j_n u \rfloor ) = \left( K_n( \lfloor j_n u \rfloor )  - \psi_{\lfloor j_n u \rfloor} (n) \right)+\Big(\psi_{\lfloor j_n u \rfloor} (n) - \sum_{r\geq 1} V_{\lfloor j_n u \rfloor - 1}(\log n - T_r) \1_{\{T_r\leq \log n\}}\Big)\\+ \sum_{r \geq 1} V_{\lfloor j_n u \rfloor - 1}(\log n - T_r) \1_{\{T_r\leq \log n\}}  =: Y_1(n,u) + Y_2(n,u) + Y_3(n,u).
\end{multline*}
It suffices to show that, as $n \to \infty$,
\begin{equation}\label{eq:limits}
\frac{ j_n^\alpha Y_i(n,u)}{\rho_{\lfloor j_n u \rfloor - 1} (\log n)^{\alpha \lfloor j_n u \rfloor}}~\overset{\mmp}{\longrightarrow}~ 0, \quad i=1,2,
\end{equation}
where $\overset{\mmp}{\longrightarrow}$ denotes convergence in probability, and
$$
\left( \frac{c j_n^\alpha Y_3(n,u)}{ \rho_{\lfloor j_n u \rfloor - 1} (\log n)^{\alpha \lfloor j_n u \rfloor}} \right)_{u>0} \overset{{\rm f.d.d.}}{\longrightarrow} \left(  \int_0^\infty \eee^{-\alpha u y} {\rm d}\mathcal{Z}_\alpha^\leftarrow(y) \right)_{u>0}.
$$

We are going to apply Theorems \ref{theorem2} and \ref{theorem3} and Lemmas \ref{lemma45} and \ref{lemma46}. While doing so, we replace $t$ with $\log n$ and choose any diverging positive function $t \mapsto j(t)$ satisfying $j(\log n) =j_n$ and $j(t)=o(t^{\min \left( \frac{1}{3}, \frac{\alpha-\beta}{\alpha -\beta +1} \right)})$ as $t\to\infty$. By Lemma \ref{lem:estim1}, Assumptions $A$ and $B$ entail \eqref{eq:estim}. With this at hand, relation \eqref{eq:limits} with $i=2$ follows from Theorem \ref{theorem3}. 
The limit relation involving $Y_3$ is ensured by Theorem \ref{theorem2}.
In view of Markov's inequality, relation \eqref{eq:limits} with $i=1$ follows if we can prove that, with $u>0$ fixed,
$$
\lim_{n \to \infty} \frac{j_n^\alpha \me |K_n(\lfloor j_n u \rfloor)-\psi_{\lfloor j_n u \rfloor}(n)|}{ \rho_{\lfloor j_n u \rfloor - 1} (\log n)^{\alpha \lfloor j_n u \rfloor}}= 0.
$$
In Section 6 of \cite{Buraczewski+Dovgay+Iksanov:2020}, see the top of p.~21, it was shown that
$$
\me |K_n(\lfloor j_n u \rfloor) - \psi_{\lfloor j_n u \rfloor}(n)|\leq n \int_{(n,\,\infty)} x^{-1}{\rm d}(\me \psi_{\lfloor j_n u \rfloor}(x)) + \int_{[1,\, n]}\eee^{-n/x} {\rm d}(\me \psi_{\lfloor j_n u \rfloor}(x)).
$$
By Lemma \ref{lemma46} applied in the particular setting $(\xi, \eta) = (|\log W|, |\log(1-W)|$, so that $V_{\lfloor j_n u \rfloor} (\log x)=\me \psi_{\lfloor j_n u \rfloor}(x)$ for $x \geq 1$,
we obtain, as $ n \to \infty$,
\begin{equation*}
n \int_{(n, \infty)} x^{-1}{\rm d}(\me \psi_{\lfloor j_n u \rfloor}(x)) = \int_{(\log n, \infty)}
e^{\log n - x} {\rm d} V_{\lfloor j_n u \rfloor} (x) = O\left(  V_{\lfloor j_n u \rfloor - 1} (\log n) \right)
\end{equation*}
for each fixed $u>0$. The function $f$ defined by $f(x)=\exp(-\eee^x)$ is decreasing and Lebesgue integrable on $[0,\infty)$. Hence, it is dRi on $[0,\infty)$, see, for instance, Lemma 6.2.1 (a) in \cite{Iksanov:2016}.
By Lemma \ref{lemma45}, with the so defined $f$, as $n\to \infty$,
\begin{equation*}
\int_{[1,\,n]} \eee^{-n/x} {\rm d}(\me \psi_{\lfloor j_n u \rfloor}(x)) = \int_{[0,\, \log n]}\exp(-\eee^{\log n - x}){\rm d} V_{\lfloor j_n u \rfloor}(x)= O\left(  V_{\lfloor j_n u \rfloor - 1} (\log n) \right)
\end{equation*}
for each fixed $u>0$. Invoking \eqref{eq:sim} and the assumption on the growth rate of $j_n$ which particularly entails $j_n=o(\log n)$ as $n\to\infty$ we infer
$$
\frac{j_n^\alpha V_{\lfloor j_n u \rfloor - 1} (\log n) }{ \rho_{\lfloor j_n u \rfloor - 1} (\log n)^{\alpha \lfloor j_n u \rfloor}}~\sim~\frac{j_n^\alpha  }{(\log n)^\alpha}~\to~ 0, \quad n\to\infty
$$
thereby arriving at \eqref{eq:limits} with $i=1$. The proof of Theorem \ref{main} is complete.
\end{proof}

\section{Appendix}

The following formula is rather standard, see, for instance, formula (A7) in \cite{Iksanov+Marynych+Meiners:2016}. 
\begin{lemma}\label{lem:formula}
Let $\gamma>0$ and $\eta$ be a positive random variable with the Laplace transform $\ell$. Then
$$
\me \eta^{-\gamma} = \frac{\gamma}{\Gamma(1+ \gamma)} \int_0^\infty s^{\gamma-1}\ell(s) {\rm d} s,
$$
where $\Gamma$ is the Euler gamma-function. Here, both sides of the equality may be infinite.
\end{lemma}

Finally, we give an estimate for the gamma-function.
\begin{lemma}\label{gammainequality}
For $x,y\geq 0$,
\begin{equation}\label{eq:gamma}
\frac{\Gamma(x+1+y)}{\Gamma(x+1)} \leq (x+1+y)^y.
\end{equation}
\end{lemma}
\begin{proof}\
If $y=0$, the inequality holds trivially. Assume that $y\in\mn$. Using the equality
\begin{equation}\label{eq:gammaeq}
\Gamma(z+1)= z\Gamma(z),\quad z >0,
\end{equation}
we obtain $$\frac{\Gamma(x+1+y)}{\Gamma(x+1)}=(x+y)\cdot\ldots\cdot (x+1)\leq (x+y)^y\leq (x+1+y)^y.$$ Assume now that $y\notin\mn$. Then, by Wendel's inequality \cite{Wendel:1948}, $$\frac{\Gamma(x+1+\{y\})}{\Gamma(x+1)}\leq (x+1)^{\{y\}},$$ where $\{y\}$ is the fractional part of $y$. This in combination with \eqref{eq:gammaeq} entails\footnote{A perusal of the proof reveals that, for $x\geq 0$, $y\geq 1$, $\frac{\Gamma(x+1+y)}{\Gamma(x+1)} \leq (x+y)^y$.} $$\frac{\Gamma(x+1+y)}{\Gamma(x+1)}=(x+y)\cdot\ldots\cdot(x+1+\{y\})\frac{\Gamma(x+1+\{y\})}{\Gamma(x+1)}\leq (x+y)^{\lfloor y\rfloor}(x+1)^{\{y\}}\leq (x+1+y)^y.$$ 
\end{proof}

\noindent {\bf Acknowledgement}. 
One of the referees outlined the way towards obtaining a sufficient condition, which is given in Lemma \ref{lem:suff}. A comment of the other referee enabled us to improve our original version of Lemma \ref{gammainequality}. We gratefully acknowledge the aforementioned referees' contributions and several additional useful comments.

\bibliographystyle{plain}

\begin{thebibliography}{10}

\bibitem{Alsmeyer+Iksanov+Marynych:2017}
G.~Alsmeyer, A.~Iksanov, and A.~Marynych.
\newblock Functional limit theorems for the number of occupied boxes in the
  {B}ernoulli sieve.
\newblock {\em Stochastic Process. Appl.}, 127(3):995--1017, 2017.

\bibitem{Basraketal:2022}
B.~Basrak, M.~Conroy, M.~Olvera-Cravioto, and Z.~Palmowski.
\newblock Importance sampling for maxima on trees.
\newblock {\em Stoch. Proc. Appl.}, 148:139--179, 2022.

\bibitem{Bertoin:2008}
J.~Bertoin.
\newblock Asymptotic regimes for the occupancy scheme of multiplicative
  cascades.
\newblock {\em Stochastic Process. Appl.}, 118(9):1586--1605, 2008.

\bibitem{billingsley}
P.~Billingsley.
\newblock {\em Convergence of probability measures}.
\newblock Wiley, second edition, 1999.

\bibitem{bingham87}
N.~Bingham, C.~Goldie, and J.~Teugels.
\newblock {\em Regular variation}.
\newblock Cambridge University Press, 1987.

\bibitem{Bohdanskyietal:2024}
V.~Bohdanskyi, V.~Bohun, A.~Marynych, and I.~Samoilenko.
\newblock Arithmetic properties of multiplicative integer-valued perturbed
  random walks.
\newblock {\em Mod. Stoch. Appl.}, 11:133--148, 2024.

\bibitem{Bohunetal:2022}
V.~Bohun, A.~Iksanov, A.~Marynych, and B.~Rashytov.
\newblock Renewal theory for iterated perturbed random walks on a general
  branching process tree: intermediate generations.
\newblock {\em J. Appl. Probab.}, 59(2):421--446, 2022.

\bibitem{Braganets+Iksanov:2023}
O.~Braganets and A.~Iksanov.
\newblock A limit theorem for a nested infinite occupancy scheme in random
  environment.
\newblock {\em Austr. J. Statist.}, 52:1--12, 2023.

\bibitem{Buraczewski+Dovgay+Iksanov:2020}
D.~Buraczewski, B.~Dovgay, and A.~Iksanov.
\newblock On intermediate levels of nested occupancy scheme in random
  environment generated by stick-breaking {I}.
\newblock {\em Electron. J. Probab.}, 25:1--24, 2020.

\bibitem{Carlsson+Nerman:1986}
H.~Carlsson and O.~Nerman.
\newblock An alternative proof of {L}orden's renewal inequality.
\newblock {\em Adv. in Appl. Probab.}, 18(4):1015--1016, 1986.

\bibitem{Carmona+Petit+Yor:1997}
P.~Carmona, F.~Petit, and M.~Yor.
\newblock Exponential functionals and principal values related to
              {B}rownian motion.
\newblock {In \em Exponential functionals and principal values related to
              {B}rownian motion}, Bibl. Rev. Mat. Iberoamericana, 73--130, 1997.



\bibitem{Duchamps+Pitman+Tang:2019}
J.-J. Duchamps, J.~Pitman, and W.~Tang.
\newblock Renewal sequences and record chains related to multiple zeta sums.
\newblock {\em Trans. Amer. Math. Soc.}, 371(8):5731--5755, 2019.

\bibitem{Gnedin+Iksanov:2020}
A.~Gnedin and A.~Iksanov.
\newblock On nested infinite occupancy scheme in random environment.
\newblock {\em Probab. Theory Related Fields}, 177:855--890, 2020.

\bibitem{Gnedin+Iksanov+Marynych:2010}
A.~Gnedin, A.~Iksanov, and A.~Marynych.
\newblock The {B}ernoulli sieve: an overview.
\newblock In {\em 21st {I}nternational {M}eeting on {P}robabilistic,
  {C}ombinatorial, and {A}symptotic {M}ethods in the {A}nalysis of {A}lgorithms
  ({A}of{A}'10)}, Discrete Math. Theor. Comput. Sci. Proc., AM, pages 329--341.
  Assoc. Discrete Math. Theor. Comput. Sci., Nancy, 2010.

\bibitem{Gnedin:2004}
A.~V. Gnedin.
\newblock The {B}ernoulli sieve.
\newblock {\em Bernoulli}, 10(1):79--96, 2004.

\bibitem{hawkes72}
J.~Hawkes.
\newblock A lower {L}ipschitz condition for the stable subordinator.
\newblock {\em Z. Wahrscheinlichkeitstheorie verw. Geb.}, 17:23--32, 1971.

\bibitem{IKSANOV2013}
A.~Iksanov.
\newblock Functional limit theorems for renewal shot noise processes with
  increasing response functions.
\newblock {\em Stoch. Proc. Appl.}, 123(6):1987--2010, 2013.

\bibitem{Iksanov:2016}
A.~Iksanov.
\newblock {\em Renewal theory for perturbed random walks and similar
  processes}.
\newblock Birkh\"{a}user/Springer, Cham, 2016.

\bibitem{Iksanov+Jedidi+Bouzzefour:2017}
A.~Iksanov, W.~Jedidi, and F.~Bouzeffour.
\newblock A law of the iterated logarithm for the number of occupied boxes in
  the {B}ernoulli sieve.
\newblock {\em Statist. Probab. Lett.}, 126:244--252, 2017.

\bibitem{Iksanov+Mallein:2022}
A.~Iksanov, and B.~Mallein.
\newblock Late levels of nested occupancy scheme in random
environment.
\newblock {\em Stoch. Models.}, 38(1):130--166, 2022.

\bibitem{Iksanov+Marynych+Meiners:2016}
A.~Iksanov, A.~Marynych, and M.~Meiners.
\newblock Moment convergence of first-passage times in renewal theory.
\newblock {\em Statist. Probab. Lett.}, 119:134--143, 2016.



\bibitem{Iksanov+Marynych+Rashytov:2022}
A.~Iksanov, A.~Marynych, and B.~Rashytov.
\newblock Stable fluctuations of iterated perturbed random walks in
  intermediate generations of a general branching process tree.
\newblock {\em Lith. Math. J.}, 62(4):447--466, 2022.

\bibitem{Iksanov+Marynych+Samoilenko:2022}
A.~Iksanov, A.~Marynych, and I.~Samoilenko.
\newblock On intermediate levels of a nested occupancy scheme in a random
  environment generated by stick-breaking {II}.
\newblock {\em Stochastics}, 94:1077--1101, 2022.

\bibitem{jap2023}
A.~Iksanov, B.~Rashytov, and I.~Samoilenko.
\newblock Renewal theory for iterated perturbed random walks on a general
  branching process tree: early generations.
\newblock {\em J. Appl. Probab.}, 60(1):45--67, 2023.

\bibitem{Joseph:2011}
A.~Joseph.
\newblock A phase transition for the heights of a fragmentation tree.
\newblock {\em Random Structures Algorithms}, 39(2):247--274, 2011.

\bibitem{Mitov+Omey:2014}
K.~V. Mitov and E.~Omey.
\newblock {\em Renewal processes}.
\newblock Springer, Cham, 2014.

\bibitem{Paulauskas:1974}
V.~I.~Paulauskas.
\newblock Estimates of the remainder term in limit theorems in the case of stable limit law.
\newblock {\em Lith. Math. J.}, 14:127–146, 1974.

\bibitem{Pitman+Tang:2019}
J.~Pitman and W.~Tang.
\newblock Regenerative random permutations of integers.
\newblock {\em Ann. Probab.}, 47(3):1378--1416, 2019.

\bibitem{Pitman+Yakubovich:2017}
J.~Pitman and Y.~Yakubovich.
\newblock Extremes and gaps in sampling from a {GEM} random discrete
  distribution.
\newblock {\em Electron. J. Probab.}, 22:Paper No. 44, 26, 2017.

\bibitem{Pitman+Yakubovich:2019}
J.~Pitman and Y.~Yakubovich.
\newblock Gaps and interleaving of point processes in sampling from a residual
  allocation model.
\newblock {\em Bernoulli}, 25(4B):3623--3651, 2019.

\bibitem{Wendel:1948}
J. G.~ Wendel.
\newblock Note on the Gamma function.
\newblock {\em Amer. Math. Monthly.}, 55(9): 563--564, 1948.

\bibitem{Wong:1976}
J.~S.~W. Wong and R.~Wong.
\newblock On asymptotic solutions of the renewal equation.
\newblock {\em J. Math. Analys. Appl.}, 53:243--250, 1976.

\end{thebibliography}

\end{document}